\documentclass{amsart}
\usepackage{amssymb}
\usepackage{amscd}
\usepackage{a4wide}
\usepackage[utf8]{inputenc}
\usepackage[english]{babel}

  \newtheorem{proposition}{Proposition}[section]
  \newtheorem{lemma}[proposition]{Lemma}
  
  \newtheorem{theorem}[proposition]{Theorem}

  \theoremstyle{definition}
  \newtheorem{definition}[proposition]{Definition}
   
  \newtheorem{example}[proposition]{Example}

  \theoremstyle{remark}
  \newtheorem{remark}[proposition]{Remark}

\begin{document}

\title{Covariant Differential Calculus Over Monoidal Hom-Hopf Algebras}
\author{SERKAN KARA\c{C}UHA}
\address{Department of Mathematics, FCUP, University of Porto, Rua Campo Alegre
687, 4169-007 Porto, Portugal}
\email{s.karacuha34@gmail.com, karacuha@itu.edu.tr}
\keywords{covariant Hom-first order differential calculus, universal Hom-differential calculus, quantum Hom-tangent space, quantum Hom-Lie algebra}

\begin{abstract}
 Concepts of first order differential calculus (FODC) on a monoidal Hom-algebra and left-covariant FODC over a left Hom-quantum space with respect to a monoidal Hom-Hopf algebra are presented. Then, extension of the universal FODC over a monoidal Hom-algebra to a universal Hom-differential calculus is described. Next, concepts of left(right)-covariant and bicovariant FODC over a monoidal Hom-Hopf algebra are studied in detail. Subsequently, notion of quantum Hom-tangent space associated to a bicovariant Hom-FODC is introduced and equipped with an analogue of Lie bracket (commutator) through Woronowicz' braiding. Finally, it is proven that this commutator satisfies quantum versions of the antisymmetry relation and Hom-Jacobi identity.
\end{abstract}

\maketitle
\section{Introduction}
Hom-type algebras were first introduced in the form of Hom-Lie algebras in \cite{HartwigLarssonSilvestrov}, where the Jacobi identity is replaced by the so-called Hom-Jacobi identity via a linear endomorphism. In 2008, Hom-associative algebras were suggested in \cite{MakhloufSilvestrov} to induce a Hom-Lie algebra using the commutator bracket. Other Hom-type structures such as Hom-coalgebras, Hom-bialgebras, Hom-Hopf algebras and their properties were further investigated and developed in \cite{AmmarMakhlouf,DekkarMakhlouf,Gohr,MakhloufSilvestrov1,MakhloufSilvestrov2,MakhloufPanaite,Yau,Yau1,Yau2,Yau5}. In \cite{CaenepeelGoyvaerts}, the authors studied the Hom-bialgebras and Hom-Hopf algebras in the context of tensor categories, and these objects are featured as {\it monoidal}. One can find further research on monoidal Hom-Hopf algebras and structures on them such as Hom-Yetter-Drinfeld modules and covariant Hom-bimodules in \cite{ChenWangZhang,ChenZhang,GuoChen,Karacuha1,LiuShen,WangGuo}.

The general theory of covariant differential calculi on quantum groups was introduced by S. L. Woronowicz in \cite{Woronowicz}, \cite{Woronowicz0},\cite{Woronowicz1}. Many results obtained in this paper in the Hom-setting follow from the classical results appear in the fundamental reference \cite{Woronowicz1}. In Section 2, after the notions of first order differential calculus (FODC) on a monoidal Hom-algebra and left-covariant FODC over a left Hom-quantum space with respect to a monoidal Hom-Hopf algebra are presented, the left-covariance of a Hom-FODC is characterized. Then, in Section 3, extension of the universal FODC over a monoidal Hom-algebra to a universal Hom-differential calculus (Hom-DC) is described as well (for the classical case, that is, for the extension of a FODC over an algebra $A$ to the differential envelope of $A$ one should refer to \cite{CuntzQuillen}, \cite{Connes1}). Thereafter, in Section 4 and Section 5, the concepts of left-covariant and bicovariant FODC over a monoidal Hom-Hopf algebra $(H,\alpha)$ are studied in detail. A subobject $\mathcal{R}$ of $ker\varepsilon$, which is a right Hom-ideal of $(H,\alpha)$, and a quantum Hom-tangent space are associated to each left-covariant $(H,\alpha)$-Hom-FODC: It is indicated that left-covariant Hom-FODCs are in one-to one correspondence with these right Hom-ideals $\mathcal{R}$, and that the quantum Hom-tangent space and the left coinvariant of the monoidal Hom-Hopf algebra on Hom-FODC form a nondegenerate dual pair. The quantum Hom-tangent space associated to a bicovariant Hom-FODC is equipped with an analogue of Lie bracket (or commutator) through Woronowicz' braiding and it is proven that this commutator satisfies quantum versions of the antisymmetry relation and Hom-Jacobi identity, which is therefore called the quantum Hom-Lie algebra of that bicovariant Hom-FODC. Throughout, we work with vector spaces over a field $k$.

\section{Left-Covariant FODC over Hom-quantum spaces}

\begin{definition} Let $(B,\beta)$ be a monoidal Hom-bialgebra. A {\it right} $(B,\beta)$-{\it Hom-comodule algebra} (or {\it Hom-quantum space}) $(A,\alpha)$ is a monoidal Hom-algebra and a right $(B,\beta)$-Hom-comodule with a Hom-coaction $\rho^{A}:A\to A\otimes B,\: a\mapsto a_{(0)}\otimes a_{(1)}$ such that $\rho^{A}$ is a Hom-algebra morphism, i.e., for any $a,a' \in A$
\begin{equation*}\label{Hom-comodule-algebra-cond}(aa')_{(0)}\otimes (aa')_{(1)}=a_{(0)}a'_{(0)}\otimes a_{(1)}{a'}_{(1)},\:\: \rho^{A}(1_A)=1_A\otimes 1_B.\end{equation*}
\end{definition}

\begin{definition}A first order differential calculus over a monoidal Hom-algebra $(A,\alpha)$ is an $(A,\alpha)$-Hom-bimodule $(\Gamma,\gamma)$ with a linear map $d: A\to \Gamma$ such that
\begin{enumerate}
\item $d$ satisfies the Leibniz rule, i.e., $d(ab)=a\cdot db+ da\cdot b, \forall a,b\in A$,
\item $d\circ \alpha=\gamma\circ d$, which means that $d$ is in $\widetilde{\mathcal{H}}(\mathcal{M}_k)$,
\item $\Gamma$ is linearly spanned by the elements of the form $(a\cdot db)\cdot c$ with $a,b,c \in A$.
\end{enumerate}
\end{definition}
We call $(\Gamma,\gamma)$ an $(A,\alpha)$-Hom-FODC for short.
\begin{remark}
\begin{enumerate}
\item In the above definition, the second condition, i.e. $d\circ \alpha=\gamma\circ d$, is equivalent to the equality $d(1)=0$.
\item By the compatibility condition for Hom-bimodule structure of $(\Gamma,\gamma)$, we have $(a\cdot db)\cdot c=\alpha(a)\cdot(db\cdot \alpha^{-1}(c))$, which implies that $\Gamma$ is also linearly spanned by the elements $a\cdot(db\cdot c)$ for all $a,b,c\in A$. Thus we denote $\Gamma=(A\cdot dA)\cdot A=A\cdot(dA\cdot A)$.
\item   By using the Leibniz rule and the fact that $d(\alpha(a))=\gamma(da)$ for any $a\in A$, we get
\begin{eqnarray*}(a\cdot db)\cdot c &=&(d(ab)-da\cdot b)\cdot c=d(ab)\cdot c- (da\cdot b)\cdot c\\
&=&d(ab)\cdot c- \gamma(d(a))\cdot (b\alpha^{-1}(c))=d(ab)\cdot c- d(\alpha(a))\cdot (b\alpha^{-1}(c)),
\end{eqnarray*}
and
\begin{eqnarray*}\alpha(a)\cdot(db\cdot \alpha^{-1}(c)) &=&\alpha(a)\cdot(d(b\alpha^{-1}(c))-b\cdot d(\alpha^{-1}(c)))\\
&=&\alpha(a)\cdot d(b\alpha^{-1}(c))-\alpha(a)\cdot(b\cdot d(\alpha^{-1}(c)))\\
&=&\alpha(a)\cdot d(b\alpha^{-1}(c))-(ab)\cdot d(c).
\end{eqnarray*}
Hence, $\Gamma=A\cdot dA=dA\cdot A$.
\end{enumerate}
\end{remark}
\begin{definition}\label{left-covariant-Hom-FODC}Let $(H,\beta)$ be a monoidal Hom-bialgebra and $(A,\alpha)$ be a left Hom-quantum space for $(H,\beta)$ (i.e. a left $(H,\beta)$-Hom-comodule algebra) with the left Hom-coaction $\varphi: A\to H\otimes A,\: a\mapsto a_{(-1)}\otimes a_{(0)}$. An $(A,\alpha)$-Hom-FODC $(\Gamma,\gamma)$ is called {\it left-covariant} with respect to $(H,\beta)$ if there is a left Hom-coaction $\phi:\Gamma\to H\otimes\Gamma,\: \omega\mapsto\omega_{(-1)}\otimes \omega_{(0)}$ of $(H,\beta)$ on $(\Gamma,\gamma)$ such that
\begin{enumerate}
\item $\phi(\alpha(a)\cdot(\omega\cdot b))=\varphi(\alpha(a))(\phi(\omega)\varphi(b))$, $\forall a,b \in A$, $\omega\in \Gamma$,
\item $\phi(da)=(id\otimes d)\varphi(a) $, $\forall a\in A$
\end{enumerate}
\end{definition}
 Condition $(1)$ can equivalently be written as $\phi((a\cdot\omega)\cdot \alpha(b))=(\varphi(a)\phi(\omega))\varphi(\alpha(b))$ by using the Hom-bimodule compatibility conditions for $(\Gamma,\gamma)$ and $(H\otimes \Gamma,\beta\otimes\gamma)$, where left and right $(H\otimes A,\beta\otimes\alpha)$-Hom-module structures of $(H\otimes \Gamma,\beta\otimes\gamma)$ are respectively given by
 $$(h\otimes a)(h'\otimes \omega)=hh'\otimes a\cdot \omega,$$
 $$(h'\otimes \omega)(h\otimes a)=hh'\otimes \omega\cdot a$$
for $h,h\in H$, $a\in A$ and $\omega \in \Gamma$. Condition $(2)$ means that $d:A\to \Gamma$ is left $(H,\beta)$-colinear, since the equality $d\circ\alpha=\gamma\circ d$ holds too.

One can see that for a given $(A,\alpha)$-Hom-FODC $(\Gamma,\gamma)$ there exists at most one morphism $\phi:\Gamma\to H\otimes\Gamma$ in $\widetilde{\mathcal{H}}(\mathcal{M}_k)$ which makes $(\Gamma,\gamma)$ left-covariant: Indeed, if there is one such $\phi$, then by the conditions $(1)$ and $(2)$ in Definition \ref{left-covariant-Hom-FODC} we do the following computation

\begin{eqnarray*}\phi\left(\sum_ia_i\cdot db_i\right)&=&\sum_i\phi(\gamma^{-1}(a_i\cdot db_i)\cdot 1_A)=\sum_i\phi((\alpha^{-1}(a_i)\cdot \gamma^{-1}(db_i))\cdot 1_A)\\
&=&\sum_i(\varphi(\alpha^{-1}(a_i))\phi(\gamma^{-1}(db_i)))\varphi(1_A)\\
&=&\sum_i([(\beta^{-1}\otimes\alpha^{-1})(\varphi(a_i))][(\beta^{-1}\otimes\gamma^{-1})(\phi(db_i))])(1_H\otimes1_A)\\
&=&\sum_i[(\beta^{-1}\otimes \gamma^{-1})(\varphi(a_i)\phi(db_i))](1_H\otimes 1_A)\\
&=&\sum_i\varphi(a_i)\phi(db_i)=\sum_i\varphi(a_i)(id\otimes d)(\varphi(b_i)),
\end{eqnarray*}
showing that $\varphi$ and $d$ describe $\phi$ uniquely.

\begin{proposition}\label{left-covariance-of-a-Hom-FODC}Let $(\Gamma,\gamma)$ be an $(A,\alpha)$-Hom-FODC. Then the following statements are equivalent:
\begin{enumerate}
\item $(\Gamma,\gamma)$ is left-covariant.
\item There is a morphism $\phi:\Gamma \to H\otimes \Gamma$ in $\widetilde{\mathcal{H}}(\mathcal{M}_k)$ such that $\phi(a\cdot db)=\varphi(a)(id\otimes d)(\varphi(b))$ for all $a,b \in A$.
\item $\sum_ia_i\cdot db_i=0$ in $\Gamma$ implies that $\sum_{i}\varphi(a_i)(id\otimes d)(\varphi(b_i))=0$ in $H\otimes \Gamma$.
\end{enumerate}
\end{proposition}
\begin{proof}$(1)\Rightarrow (2)$ and $(2)\Rightarrow (3)$ are trivial.

$(3)\Rightarrow (1):$ Let $\phi:\Gamma\to H\otimes\Gamma$ be defined by the equation

$$\phi(\sum_ia_i\cdot db_i)=\sum_i\varphi(a_i)(id\otimes d)(\varphi(b_i))$$

as was obtained in the above computation. By using hypothesis $(3)$ it is immediate to see that $\phi$ is well-defined. If we write $\varphi(a)=a_{(-1)}\otimes a_{(0)}$ for any $a\in A$ and $\phi(\omega)=\omega_{(-1)}\otimes \omega_{(0)}$ for all $\omega\in \Gamma$, then for $\omega=\sum_{i}a_{i}\cdot db_{i} \in \Gamma$ we have
$$\phi(\omega)=\omega_{(-1)}\otimes \omega_{(0)}=\sum_{i}a_{i,(-1)}b_{i,(-1)}\otimes a_{i,(0)}\cdot db_{i,(0)},$$
where we have used the notation $\varphi(a_{i})=a_{i,(-1)}\otimes a_{i,(0)}$. Now we prove that $\phi$ is a left Hom-coaction of $(H,\beta)$ on $(\Gamma,\gamma)$:

\begin{eqnarray*}\beta^{-1}(\omega_{(-1)})\otimes \phi(\omega_{(0)})&=&\sum_{i}\beta^{-1}(a_{i,(-1)}b_{i,(-1)})\otimes \phi(a_{i,(0)}\cdot db_{i,(0)})\\ \\
&=&\sum_{i}\beta^{-1}(a_{i,(-1)})\beta^{-1}(b_{i,(-1)})\otimes a_{i,(0)(-1)}b_{i,(0)(-1)}\otimes a_{i,(0)(0)}\cdot db_{i,(0)(0)}\\ \\
&=&\sum_{i}a_{i,(-1)1}b_{i,(-1)1}\otimes a_{i,(-1)2}b_{i,(-1)2}\otimes \alpha^{-1}(a_{i,(0)})\cdot d(\alpha^{-1}(b_{i,(0)}))\\ \\
&=&\sum_{i}(a_{i,(-1)}b_{i,(-1)})_{1}\otimes (a_{i,(-1)}b_{i,(-1)})_{2}\otimes \gamma^{-1}(a_{i,(0)}\cdot db_{i,(0)})\\ \\
&=&\Delta(\omega_{(-1)})\otimes \gamma^{-1}(\omega_{(0)}),
\end{eqnarray*}

\begin{eqnarray*}\varepsilon(\omega_{(-1)})\omega_{(0)}&=&\sum_{i}\varepsilon(a_{i,(-1)}b_{i,(-1)})a_{i,(0)}\cdot db_{i,(0)}\\ \\
&=&\sum_{i}\varepsilon(a_{i,(-1)})a_{i,(0)}\cdot d(\varepsilon(b_{i,(-1)})b_{i,(0)})\\ \\
&=&\sum_{i}\alpha^{-1}(a_i)\cdot d(\alpha^{-1}(b_i))=\gamma^{-1}(\omega),
\end{eqnarray*}
\begin{eqnarray*}\phi(\gamma(\sum_ia_i\cdot db_i))&=&\phi(\sum_{i}\alpha(a_i)\cdot d(\alpha(b_i)))\\ \\
&=&\sum_i\varphi(\alpha(a_i))(id\otimes d)(\varphi(\alpha(b_i)))\\ \\
&=&(\beta\otimes \alpha)(\varphi(a_i))(id\otimes d)((\beta\otimes \alpha)(\varphi(b_i)))\\ \\
&=&\sum_i\beta(a_{i,(-1)})\beta(b_{i,(-1)})\otimes \alpha(a_{i,(0)})\cdot d(\alpha(b_{i,(0)}))\\ \\
&=&(\beta\otimes \gamma)(\phi(\sum_ia_i\cdot db_i)).
\end{eqnarray*}

Let $\omega=\sum_ia_i\cdot db_i \in \Gamma$, and $a,b \in A$. Then we have
\begin{eqnarray*}\lefteqn{\phi(\alpha(a)\cdot(\omega\cdot b))}\hspace{2em}\\
&=&\phi(\alpha(a)\cdot(\sum_i(a_i\cdot db_i)\cdot b))\\
&=&\phi(\alpha(a)\cdot(\sum_i(\alpha(a_i)\cdot d(b_i \alpha^{-1}(b))-(a_ib_i)\cdot db)))\\
&=&\phi(\sum_i[(a\alpha(a_i))\cdot d(\alpha(b_i)b)-(a(a_ib_i))\cdot d(\alpha(b))])\\
&=&\sum_i\varphi(a\alpha(a_i))(id\otimes d)(\varphi(\alpha(b_i)b))-\sum_i\varphi(a(a_ib_i))(id\otimes d)(\varphi(\alpha(b)))\\
&=&\sum_i(\varphi(a)\varphi(\alpha(a_i)))(id\otimes d)(\varphi(\alpha(b_i)b))-\sum_i(\varphi(a)\varphi(a_ib_i))(id\otimes d)(\varphi(\alpha(b)))\\
&=&\sum_i\varphi(\alpha(a))(\varphi(\alpha(a_i))(id\otimes d)(\varphi(b_i\alpha^{-1}(b))))-\sum_i\varphi(\alpha(a))(\varphi(a_ib_i)(id\otimes d)(\varphi(b)))\\
&=&\varphi(\alpha(a))(\sum_i\varphi(\alpha(a_i))(id\otimes d)(\varphi(b_i\alpha^{-1}(b)))-\sum_i\varphi(a_ib_i)(id\otimes d)(\varphi(b)))\\
&=&\varphi(\alpha(a))(\sum_i\varphi(\alpha(a_i))(id\otimes d)(\varphi(b_i\alpha^{-1}(b)))-\sum_i(\varphi(a_i)\varphi(b_i))(id\otimes d)(\varphi(b)))\\
&=&\varphi(\alpha(a))(\sum_i\varphi(\alpha(a_i))[(id\otimes d)(\varphi(b_i\alpha^{-1}(b)))-\varphi(b_i)(id\otimes d)(\varphi(\alpha^{-1}(b)))])\\
&=&\varphi(\alpha(a))(\sum_i\varphi(\alpha(a_i))[((id\otimes d)(\varphi(b_i)))\varphi(\alpha^{-1}(b))])\\
&=&\varphi(\alpha(a))([\sum_i\varphi(a_i)(id\otimes d)(\varphi(b_i))]\varphi(b))\\
&=&\varphi(\alpha(a))(\phi(\omega)\varphi(b)),
\end{eqnarray*}
which is the first condition of Definition \ref{left-covariant-Hom-FODC}.
For any $a\in A$, we get
\begin{eqnarray*}\phi(da)&=&\phi(1_A\cdot\gamma^{-1}(da))=\phi(1_A\cdot d(\alpha^{-1}(a)))\\
&=&\varphi(1_A)(id\otimes d)(\varphi(\alpha^{-1}(a)))\\
&=&(1_H\otimes 1_A)[((id\otimes d)\circ(\beta^{-1}\otimes\alpha^{-1}))(\varphi(a))]\\
&=&(1_H\otimes 1_A)[((\beta^{-1}\otimes\gamma^{-1})\circ(id\otimes d))(\varphi(a))]=(id\otimes d)(\varphi(a)),
\end{eqnarray*}
which is the second condition of Definition \ref{left-covariant-Hom-FODC}.
\end{proof}

\section{Universal Differential Calculus of a Monoidal Hom-Algebra}
In the theory of quantum groups, a differential calculus is a substitute of the de Rham complex of a smooth manifold for arbitrary algebras. In this section, the definition of differential calculus over a monoidal Hom-algebra (abbreviated, Hom-DC) is given and the construction of the universal differential calculus of a monoidal Hom-algebra (universal Hom-DC) is outlined.
\begin{definition}
A {\it graded} monoidal Hom-algebra is a monoidal Hom-algebra $(A,\alpha)$ together with subobjects $A_n, n\geq 0$ (that is, for each $k$-submodule $A_n\subseteq A$, $(A_n,\alpha|_{A_n})\in \widetilde{\mathcal{H}}(\mathcal{M}_k) $) such that
    $$A=\bigoplus_{n\geq 0}A_n,$$
$1\in A_0$, and $A_nA_m\subseteq A_{n+m}$ for all $n,m\geq 0$.
\end{definition}

\begin{definition}
 A differential calculus over a monoidal Hom-algebra $(A,\alpha)$ is a graded monoidal Hom-algebra $(\Gamma=\bigoplus_{n\geq 0}\Gamma^n, \gamma)$ with a linear map $d:\Gamma\to \Gamma$, in $\widetilde{\mathcal{H}}(\mathcal{M}_k)$, of degree one (i.e., $d:\Gamma^n\to \Gamma^{n+1}$) such that
    \begin{enumerate}
    \item $d^2=0$,
    \item $d(\omega\omega')=d(\omega)\omega'+(-1)^n\omega d(\omega')$ for $\omega\in \Gamma^n,\omega'\in \Gamma$ (graded Leibniz rule),
    \item $\Gamma^0=A$, $\gamma|_{\Gamma^0}=\alpha$, and $\Gamma^n$ is a linear span of the elements of the form

    $a_0(da_1(\cdots(da_{n-1}da_n)\cdots))$ with $a_0,\cdots, a_n \in A$, $n\geq 0$.
    \end{enumerate}
A {\it differential Hom-ideal} of $(\Gamma,\gamma)$ is a Hom-ideal $\mathcal{I}$ of the monoidal Hom-algebra $(\Gamma,\gamma)$ (that is, $\mathcal{I}$ is a subobject of $(\Gamma,\gamma)$ such that $(\Gamma\mathcal{I})\Gamma=\Gamma(\mathcal{I}\Gamma)\subset\mathcal{I}$) such that $\mathcal{I}\cap \Gamma^0=\{0\}$ and $\mathcal{I}$ is invariant under the differentiation $d$.
\end{definition}
Let us write $\gamma^n$ for $\gamma|_{\Gamma^n}$ for all $n\geq 0$. Then, the map $d \in \widetilde{\mathcal{H}}(\mathcal{M}_k)$ means that $d\circ \gamma^n=\gamma^{n+1}\circ d$ for all $ n\geq 0$. Let $\mathcal{I}$ be a differential Hom-ideal of a $(A,\alpha)$-Hom-DC $(\Gamma,\gamma)$. Then, $\gamma$ induces an automorphism $\bar{\gamma}$ of $\bar{\Gamma}:=\Gamma/\mathcal{I}$ and $(\bar{\Gamma},\bar{\gamma})$ is a monoidal Hom-algebra. Since the condition $\mathcal{I}\cap \Gamma^0=\{0\}$ holds, $\bar{\Gamma}^0=\Gamma^0=A$. On the other hand, let $\pi:\Gamma\to \bar{\Gamma}$ be the canonical surjective map and define $\bar{d}:\bar{\Gamma}\to \bar{\Gamma}$ by $\bar{d}(\pi(\omega)):=\pi(d(\omega))$ for any $\omega\in \Gamma$. Thus, $(\bar{\Gamma},\bar{\gamma})$ is again a Hom-DC on $(A,\alpha)$ with differentiation $\bar{d}$.

In the rest of the section, the construction of the universal differential calculus on a monoidal Hom-algebra $(A,\alpha)$ is discussed. Let $(A,\alpha)$ be a monoidal Hom-algebra with Hom-multiplication $m_A:A\otimes A\to A$. The linear map $d:A\to A\otimes A$, in $\widetilde{\mathcal{H}}(\mathcal{M}_k)$, given by

$$da:=1\otimes \alpha^{-1}(a)-\alpha^{-1}(a)\otimes 1,\:\: \forall a\in A$$

satisfies the Leibniz rule: For $a,b\in A$,
\begin{eqnarray*}a\cdot db+da\cdot b&=&a\cdot(1\otimes \alpha^{-1}(b)-\alpha^{-1}(b)\otimes 1)+(1\otimes \alpha^{-1}(a)-\alpha^{-1}(a)\otimes 1)\cdot b\\
&=&\alpha^{-1}(a)1\otimes \alpha(\alpha^{-1}(b))-\alpha^{-1}(a)\alpha^{-1}(b)\otimes 1\\
&+&1\otimes \alpha^{-1}(a)\alpha^{-1}(b)-\alpha(\alpha^{-1}(a))\otimes 1\alpha^{-1}(b)\\
&=&a\otimes b-\alpha^{-1}(ab)\otimes 1+1\otimes \alpha^{-1}(ab)-a\otimes b=1\otimes \alpha^{-1}(ab)-\alpha^{-1}(ab)\otimes 1\\
&=&d(ab).
\end{eqnarray*}
For any $a\in A$, we get
\begin{eqnarray*}(d\circ \alpha)(a)&=&d(\alpha(a))=1\otimes a-a\otimes 1\\
&=&(\alpha\otimes \alpha)(1\otimes \alpha^{-1}(a)-\alpha^{-1}(a)\otimes 1)=(\alpha\otimes\alpha)(da),
\end{eqnarray*}
meaning $d$ is in $\widetilde{\mathcal{H}}(\mathcal{M}_k)$. Let $\Omega^1(A)$ be the $(A,\alpha)$-Hom-subbimodule of $(A\otimes A,\alpha\otimes \alpha)$ generated by elements of the form $a\cdot db$ for $a,b\in A$. Then we have

$$\Omega^1(A)=ker\: m_A.$$

Indeed, if $a\cdot db \in \Omega^1(A)$, then
$$m_A(a\cdot db)=m_A(a\otimes b-\alpha^{-1}(ab)\otimes 1)=ab-\alpha^{-1}(ab)1=0.$$
On the other hand, if $\sum_ia_i\otimes b_i \in ker\: m_A$ ($\sum_i$ denotes a finite sum), then $\sum_i a_ib_i=0$, thus we write

$$\sum_ia_i\otimes b_i=\sum_i(a_i\otimes b_i-\alpha^{-1}(a_ib_i)\otimes 1)=\sum_ia_i\cdot(1\otimes \alpha^{-1}(b_i)-\alpha^{-1}(b_i)\otimes 1)=\sum_ia_i\cdot db_i.$$

The left and right $(A,\alpha)$-Hom-module structures of $(\Omega^1(A),\beta)=(\Omega^1(A),(\alpha\otimes \alpha)|_{ker\:m_A})$ are respectively given by
$$a\cdot(b\cdot dc)=(\alpha^{-1}(a)b)\cdot d(\alpha(c))\qquad (a\cdot db)\cdot c=\alpha(a)\cdot d(b\alpha^{-1}(c))-(ab)\cdot dc,$$
for any $a,b,c\in A$. $(\Omega^1(A),\beta)$ is called the {\it universal first order differential calculus} of monoidal Hom-algebra $(A,\alpha)$.

Let $\bar{A}:=A/k\cdot 1$ be the quotient space of $A$ by the scalar multiples of the Hom-unit and let $\bar{a}$ denote the equivalence class $a+k\cdot 1$ for any $a\in A$. $\alpha$ induces an automorphism $\bar{\alpha}:\bar{A}\to \bar{A},\: \bar{a}\mapsto \bar{\alpha}(\bar{a})=\overline{\alpha(a)}$ and $(\bar{A},\bar{\alpha}) \in \widetilde{\mathcal{H}}(\mathcal{M}_k)$. Let  $A\otimes \bar{A}=\Omega^1(A)$ by the identification $a_0\otimes \overline{a_1}\mapsto a_0da_1$. This identification is well-defined since $d1=0$, and one can easily show that it is an $(A,\alpha)$-Hom-bimodule isomorphism once the Hom-bimodule structure of $(A\otimes \bar{A},\alpha\otimes\bar{\alpha})$ is given by, for $b\in A$,
$$b(a_0\otimes\overline{a_1})=\alpha^{-1}(b)a_0\otimes \overline{\alpha(a_1)},\:\: (a_0\otimes\overline{a_1})b=\alpha(a_0)\otimes \overline{a_1\alpha^{-1}(b)}-a_0a_1\otimes \bar{b}.$$

Now, we set
$$\Omega^n(A):=\otimes^{(n)}_{A}(\Omega^1(A))=\Omega^1(A)\otimes_A(\otimes^{(n-1)}_{A}(\Omega^1(A))).$$

Above, $\otimes^{(n)}_{A}(\Omega^1(A))$ has been put for $$T^n_A(\Omega^1(A))=\otimes^{t^n}_A(\Omega^1(A),\cdots,\Omega^1(A))=\Omega^1(A)\otimes_A(\Omega^1(A)\otimes_A(\cdots(\Omega^1(A)\otimes_A \Omega^1(A))\cdots )),$$
where $t^n$ is a fixed element in the set $T_n$ of planar binary trees with $n$ leaves and one root, which corresponds to the parenthesized monomial $x_1(x_2(\cdots(x_{n-1}x_n)\cdots))$ in $n$ noncommuting variables (see \cite{Yau0} e.g.). One should also refer to \cite[Section 6]{CaenepeelGoyvaerts} for the construction of tensor Hom-algebra applied to an object $(M,\mu)\in \widetilde{\mathcal{H}}(\mathcal{M}_k)$). So, we have, for any $n\geq 0$,

$$\Omega^n(A)=A\otimes(\otimes^{(n)}(\bar{A}))=A\otimes(\bar{A}\otimes(\bar{A}\otimes(\cdots(\bar{A}\otimes\bar{A})\cdots)))$$

by the correspondence $(A\otimes\bar{A})\otimes_A(A\otimes(\otimes^{(n-1)}(\bar{A})))=A\otimes(\otimes^{(n)}_A(\bar{A}))$, in
$\widetilde{\mathcal{H}}(\mathcal{M}_k)$,

\begin{eqnarray*}(a_0\otimes \overline{a_1})\otimes_A(a_2\otimes(\otimes^{(n-1)}(\overline{a_3},\cdots,\overline{a_{n+1}})))&\mapsto&
\alpha(a_0)\otimes(\otimes^{n}(\overline{\alpha^{-1}(a_1a_2)},\overline{a_3},\cdots,\overline{a_{n+1}}))\\
&&-a_0a_1\otimes(\otimes^{n}(\overline{a_2},\cdots,\overline{a_{n+1}})),
\end{eqnarray*}
where we have used the notation $\otimes^{(n)}(\overline{a_1},\cdots,\overline{a_{n}})$ for $\overline{a_1}\otimes(\overline{a_2}\otimes(\cdots(\overline{a_{n-1}}\otimes\overline{a_n})\cdots))$.
To the object $A\otimes(\otimes^{(n)}(\bar{A}))$ we associate the automorphism $\alpha\otimes(\otimes^{(n)}(\bar{\alpha})):A\otimes(\otimes^{(n)}(\bar{A}))\to A\otimes(\otimes^{(n)}(\bar{A}))$ given by

$$a_0\otimes(\otimes^{(n)}(\overline{a_1},\cdots,\overline{a_{n}}))\mapsto \alpha(a_0)\otimes(\otimes^{(n)}(\overline{\alpha(a_1)},\cdots,\overline{\alpha(a_{n})})),$$
for $a_0\in A$ and $\overline{a_i}\in \bar{A}$, $i=1,\cdots, n$.

On $\bigoplus_{n=0}^{\infty}\Omega^n(A)$, we define the differential by the linear mapping $d:A\otimes(\otimes^{(n)}(\bar{A}))\to A\otimes(\otimes^{(n+1)}(\bar{A}))$ of degree one by
\begin{equation}d(a_0\otimes(\overline{a_1}\otimes(\cdots(\overline{a_{n-1}}\otimes\overline{a_{n}})\cdots)))
=1\otimes(\overline{\alpha^{-1}(a_0)}\otimes(\cdots(\overline{\alpha^{-1}(a_{n-1})}\otimes\overline{\alpha^{-1}(a_{n})})\cdots)).
\end{equation}
We immediately obtain $d^{2}=0$ from the fact that $\bar{1}=0$. If we start with $a_n\in A$, multiplying on the left and applying $d$ repeatedly gives us the following
$$a_0\otimes(\overline{a_1}\otimes(\cdots(\overline{a_{n-1}}\otimes\overline{a_{n}})\cdots))=a_0(da_1(da_2(\cdots(da_{n-1}da_n)\cdots))),$$
where $a_0(da_1(da_2(\cdots(da_{n-1}da_n)\cdots)))=a_0\otimes_A(da_1\otimes_A(da_2\otimes_A(\cdots(da_{n-1}\otimes_A da_n)\cdots)))$.

We make $\bigoplus_{n=0}^{\infty}\Omega^n(A)$ an $(A,\alpha)$-Hom-bimodule as follows. The left $(A,\alpha)$-Hom-module structure is given by, for $b\in A$ and $a_0(da_1(da_2(\cdots(da_{n-1}da_n)\cdots)))\in \Omega^n(A)$, $n\geq 1$,

\begin{eqnarray*}\lefteqn{b(a_0(da_1(da_2(\cdots(da_{n-1}da_n)\cdots))))}\hspace{7em}\\
&=&(\alpha^{-1}(b)a_0)(d(\alpha(a_1))(d(\alpha(a_2))(\cdots(d(\alpha(a_{n-1}))d(\alpha(a_n))\cdots))).
\end{eqnarray*}
We now get the right $(A,\alpha)$-Hom-module structure: One can show that, for $b\in A$, $a_0da_1\in \Omega^1(A)$, $a_0(da_1da_2)\in \Omega^2(A)$ and $a_0(da_1(da_2da_3))\in \Omega^3(A)$, the following equations hold:

$$(a_0da_1) b=\alpha(a_0)d(a_1\alpha^{-1}(b))-(a_0a_1)db,$$

\begin{eqnarray*}(a_0(da_1da_2))b&=&\alpha(a_0)(d(\alpha(a_1))d(a_2\alpha^{-2}(b)))-\alpha(a_0)(d(a_1a_2)d(\alpha^{-1}(b)))\\
&+&(a_0\alpha(a_1))(d(\alpha(a_2))d(\alpha^{-1}(b))),
\end{eqnarray*}
\begin{eqnarray*}\lefteqn{(a_0(da_1(da_2da_3)))b}\hspace{2em}\\
&=&\alpha(a_0)(d(\alpha(a_1))(d(\alpha(a_2))d(a_3\alpha^{-3}(b))))-\alpha(a_0)(d(\alpha(a_1))(d(a_2a_3)d(\alpha^{-2}(b))))\\ \\
&+&\alpha(a_0)(d(a_1\alpha(a_2))(d(\alpha(a_3))d(\alpha^{-2}(b))))-(a_0\alpha(a_1))(d(\alpha^{2}(a_2))(d(\alpha(a_3))d(\alpha^{-2}(b)))).
\end{eqnarray*}
By induction, one can also prove that the equation

\begin{eqnarray*}&&(a_0(da_1(da_2(\cdots(da_{n-1}da_n)\cdots))))b\\ \\
&=&(-1)^{n}(a_0\alpha(a_1))(d(\alpha^{2}(a_2))(d(\alpha^{2}(a_3))(\cdots d(\alpha^{2}(a_{n-1}))(d(\alpha(a_n))d(\alpha^{-(n-1)}(b)))\cdots)))\\ \\
&+&\sum_{i=1}^{n-3}(-1)^{n-i}\alpha(a_0)(d(\alpha(a_1))(\cdots d(\alpha(a_{i-1}))(d(a_i\alpha(a_{i+1}))(d(\alpha^{2}(a_{i+2}))(\cdots \\
&&d(\alpha^{2}(a_{n-1}))(d(\alpha(a_n))d(\alpha^{-(n-1)}(b)))\cdots)))\cdots))\\ \\
&+&\alpha(a_0)(d(\alpha(a_1))(\cdots d(\alpha(a_{n-3}))(d(a_{n-2}\alpha(a_{n-1}))(d(\alpha(a_n))d(\alpha^{-(n-1)}(b))))\cdots))\\ \\
&-&\alpha(a_0)(d(\alpha(a_1))(\cdots d(\alpha(a_{n-2}))(d(a_{n-1}a_n)d(\alpha^{-(n-1)}(b)))\cdots))\\ \\
&+&\alpha(a_0)(d(\alpha(a_1))(\cdots (d(\alpha(a_{n-1}))d(a_n\alpha^{-n}(b)))\cdots))
\end{eqnarray*}
holds for $a_0(da_1(da_2(\cdots(da_{n-1}da_n)\cdots)))\in \Omega^{n}(A)$, $n\geq 4$.

Next, we define the Hom-multiplication between any two parenthesized monomials, by using the right Hom-module structure given above, as

\begin{eqnarray}\label{Hom-multiplication-of-monomials}\nonumber\lefteqn{[a_0(da_1(\cdots(da_{n-1}da_n)\cdots))]
[a_{n+1}(da_{n+2}(\cdots(da_{n+k-1}da_{n+k})\cdots))]}\hspace{2em}\\
&=&[(\alpha^{-1}(a_0)(d(\alpha^{-1}(a_1))(\cdots(d(\alpha^{-1}(a_{n-1}))d(\alpha^{-1}(a_n))\cdots))))a_{n+1}]\\
\nonumber&&[d(\alpha(a_{n+2}))(\cdots(d(\alpha(a_{n+k-1}))d(\alpha(a_{n+k}))\cdots))],
\end{eqnarray}
for $\omega_n=a_0(da_1(\cdots(da_{n-1}da_n)\cdots))\in \Omega^n(A)$ and $\omega_{k-1}=a_{n+1}(da_{n+2}(\cdots(da_{n+k-1}da_{n+k})\cdots))$ $\in \Omega^{k-1}(A)$. For any $n\geq 4$, we explicitly write the above multiplication:

\begin{eqnarray*}&&\omega_n\omega_{k-1}\\ \\
&=&[a_0(da_1(\cdots(da_{n-1}da_n)\cdots))][a_{n+1}(da_{n+2}(\cdots(da_{n+k-1}da_{n+k})\cdots))]\\ \\
&=&[(\alpha^{-1}(a_0)(d(\alpha^{-1}(a_1))(\cdots(d(\alpha^{-1}(a_{n-1}))d(\alpha^{-1}(a_n))\cdots))))a_{n+1}]\\
&&[d(\alpha(a_{n+2}))(\cdots(d(\alpha(a_{n+k-1}))d(\alpha(a_{n+k}))\cdots))]\\ \\
&=&[(-1)^{n}(\alpha^{-1}(a_0)a_1)(d(\alpha(a_2))(d(\alpha(a_3))(\cdots d(\alpha(a_{n-1}))(d(a_n)d(\alpha^{-(n-1)}(a_{n+1})))\cdots)))\\
&+&\sum_{i=1}^{n-3}(-1)^{n-i}a_0(d(a_1)(\cdots d(a_{i-1})(d(\alpha^{-1}(a_i)a_{i+1})(d(\alpha(a_{i+2}))(\cdots \\
&& d(\alpha(a_{n-1}))(d(a_n)d(\alpha^{-(n-1)}(a_{n+1})))\cdots)))\cdots))\\
&+&a_0(d(a_1)(\cdots d(a_{n-3})(d(\alpha^{-1}(a_{n-2})a_{n-1})(d(a_n)d(\alpha^{-(n-1)}(a_{n+1}))))\cdots))\\
&-&a_0(d(a_1)(\cdots d(a_{n-2})(d(\alpha^{-1}(a_{n-1}a_n))d(\alpha^{-(n-1)}(a_{n+1})))\cdots))\\
&+&a_0(d(a_1)(\cdots (d(a_{n-1})d(\alpha^{-1}(a_n)\alpha^{-n}(a_{n+1})))\cdots)) ]\\
&&[d(\alpha(a_{n+2}))(\cdots(d(\alpha(a_{n+k-1}))d(\alpha(a_{n+k}))\cdots))]\\ \\
&=&(-1)^{n}(a_0\alpha(a_1))(d(\alpha^{2}(a_2))(\cdots d(\alpha^{2}(a_{n-1}))(d(\alpha(a_n))(d(\alpha^{-(n-1)}(a_{n+1}))(\cdots \\
&&(d(\alpha^{-(n-1)}(a_{n+k-1}))d(\alpha^{-(n-1)}(a_{n+k})))\cdots)))\cdots))\\ \\
&+&\sum_{i=1}^{n-3}(-1)^{n-i}\alpha(a_0)(d(\alpha(a_1))(\cdots d(\alpha(a_{i-1}))(d(a_i\alpha(a_{i+1}))(d(\alpha^2(a_{i+2}))(\cdots d(\alpha^2(a_{n-1})) \\
&&(d(\alpha(a_n))(d(\alpha^{-(n-1)}(a_{n+1}))(\cdots(d(\alpha^{-(n-1)}(a_{n+k-1}))d(\alpha^{-(n-1)}(a_{n+k})))\cdots)))\cdots)))\cdots))\\ \\
&+&\alpha(a_0)(d(\alpha(a_1))(\cdots d(\alpha(a_{n-3}))(d(a_{n-2}\alpha(a_{n-1}))(d(\alpha(a_n))(d(\alpha^{-(n-1)}(a_{n+1}))(\cdots \\
&& (d(\alpha^{-(n-1)}(a_{n+k-1}))d(\alpha^{-(n-1)}(a_{n+k})))\cdots ))))\cdots))\\ \\
&-&\alpha(a_0)(d(\alpha(a_1))\cdots d(\alpha(a_{n-2}))(d(a_{n-1}a_{n-1})(d(\alpha^{-(n-1)}(a_{n+1}))(\cdots\\
&& (d(\alpha^{-(n-1)}(a_{n+k-1}))d(\alpha^{-(n-1)}(a_{n+k})))\cdots ))\cdots))\\ \\
&+&\alpha(a_0)(d(\alpha(a_1))(\cdots d(\alpha(a_{n-2}))(d(\alpha(a_{n-1}))(d(\alpha^{-1}(a_{n})\alpha^{-n}(a_{n+1}))(d(\alpha^{-(n-1)}(a_{n+2}))(\cdots\\
&& (d(\alpha^{-(n-1)}(a_{n+k-1}))d(\alpha^{-(n-1)}(a_{n+k})))\cdots ))))\cdots))
\end{eqnarray*}

On the other hand, we have the following computations for $\omega_n$ and $\omega_{k-1}$ given above:
\begin{eqnarray*} &&d\omega_n\omega_{k-1}\\
&=&[da_0(da_1(\cdots(da_{n-1}da_n)\cdots))][a_{n+1}(da_{n+2}(\cdots(da_{n+k-1}da_{n+k})\cdots))]\\
&=&[(d(\alpha^{-1}(a_0))(d(\alpha^{-1}(a_1))(\cdots(d(\alpha^{-1}(a_{n-1}))d(\alpha^{-1}(a_n))\cdots))))a_{n+1}]\\
&&[d(\alpha(a_{n+2}))(\cdots(d(\alpha(a_{n+k-1}))d(\alpha(a_{n+k})))\cdots)]\\
&=&(-1)^{n+1}\alpha(a_0)(d(\alpha(a_1))(\cdots(d(\alpha(a_{n-1}))(da_n(d(\alpha^{-n}(a_{n+1}))(\cdots\\
&&(d(\alpha^{-n}(a_{n+k-1}))d(\alpha^{-n}(a_{n+k})))\cdots))))\cdots))\\ \\
&+&(-1)^{n}d(a_0\alpha(a_1))(d(\alpha^{2}(a_{2}))(\cdots(d(\alpha^{2}(a_{n-1}))(d(\alpha(a_n))(d(\alpha^{-(n-1)}(a_{n+1}))(\cdots\\
&&(d(\alpha^{-(n-1)}(a_{n+k-1}))d(\alpha^{-(n-1)}(a_{n+k})))\cdots))))\cdots))\\ \\
&+&\sum_{i=2}^{n-2}(-1)^{n+1-i}d(\alpha(a_0))(\cdots(d(\alpha(a_{i-2}))(d(a_{i-1}\alpha(a_i))(d(\alpha^{2}(a_{i+1}))(\cdots(d(\alpha^{2}(a_{n-1}))\\
&&(d(\alpha(a_n))(d(\alpha^{-(n-1)}(a_{n+1}))(\cdots(d(\alpha^{-(n-1)}(a_{n+k-1}))d(\alpha^{-(n-1)}(a_{n+k})))\cdots))))\cdots))))\cdots)\\ \\
&+&d(\alpha(a_0))(\cdots(d(\alpha(a_{n-3}))(d(a_{n-2}\alpha(a_{n-1}))(d(\alpha(a_n))(d(\alpha^{-(n-1)}(a_{n+1}))(\cdots\\
&&(d(\alpha^{-(n-1)}(a_{n+k-1}))d(\alpha^{-(n-1)}(a_{n+k})))\cdots)))))\cdots)\\ \\
&-&d(\alpha(a_0))(\cdots(d(\alpha(a_{n-2}))(d(a_{n-1}a_n)(d(\alpha^{-(n-1)}(a_{n+1}))(\cdots\\
&&(d(\alpha^{-(n-1)}(a_{n+k-1}))d(\alpha^{-(n-1)}(a_{n+k})))\cdots))))\cdots)\\ \\
&+&d(\alpha(a_0))(\cdots(d(\alpha(a_{n-1}))(d(\alpha^{-1}(a_n)\alpha^{-n}(a_{n+1}))(d(\alpha^{-(n-1)}(a_{n+2}))(\cdots\\
&&(d(\alpha^{-(n-1)}(a_{n+k-1}))d(\alpha^{-(n-1)}(a_{n+k})))\cdots))))\cdots)
\end{eqnarray*}
and
\begin{eqnarray*} &&\omega_nd\omega_{k-1}\\
&=&[a_0(da_1(\cdots(da_{n-1}da_n)\cdots))][da_{n+1}(\cdots(da_{n+k-1}da_{n+k})\cdots)]\\
&=&\alpha(a_0)([da_1(\cdots(da_{n-1}da_n)\cdots)][d(\alpha^{-1}(a_{n+1}))(\cdots(d(\alpha^{-1}(a_{n+k-1}))d(\alpha^{-1}(a_{n+k})))\cdots)])\\
&=&\alpha(a_0)(d(\alpha(a_1))(\cdots(d(\alpha(a_{n-1}))(d(a_n)(d(\alpha^{-n}(a_{n+1}))(\cdots\\
&&(d(\alpha^{-n}(a_{n+k-1}))d(\alpha^{-n}(a_{n+k})))\cdots))))\cdots)).
\end{eqnarray*}
Thus, the equation below holds:

\begin{eqnarray*}&&d\omega_n\omega_{k-1}+(-1)^n\omega_nd\omega_{k-1}\\
&=&(-1)^{n}d(a_0\alpha(a_1))(d(\alpha^{2}(a_{2}))(\cdots(d(\alpha^{2}(a_{n-1}))(d(\alpha(a_n))(d(\alpha^{-(n-1)}(a_{n+1}))(\cdots\\
&&(d(\alpha^{-(n-1)}(a_{n+k-1}))d(\alpha^{-(n-1)}(a_{n+k})))\cdots))))\cdots))\\
&+&\sum_{i=2}^{n-2}(-1)^{n+1-i}d(\alpha(a_0))(\cdots(d(\alpha(a_{i-2}))(d(a_{i-1}\alpha(a_i))(d(\alpha^{2}(a_{i+1}))(\cdots(d(\alpha^{2}(a_{n-1}))\\
&&(d(\alpha(a_n))(d(\alpha^{-(n-1)}(a_{n+1}))(\cdots(d(\alpha^{-(n-1)}(a_{n+k-1}))d(\alpha^{-(n-1)}(a_{n+k})))\cdots))))\cdots))))\cdots)\\
&+&d(\alpha(a_0))(\cdots(d(\alpha(a_{n-3}))(d(a_{n-2}\alpha(a_{n-1}))(d(\alpha(a_n))(d(\alpha^{-(n-1)}(a_{n+1}))(\cdots\\
&&(d(\alpha^{-(n-1)}(a_{n+k-1}))d(\alpha^{-(n-1)}(a_{n+k})))\cdots)))))\cdots)\\
&-&d(\alpha(a_0))(\cdots(d(\alpha(a_{n-2}))(d(a_{n-1}a_n)(d(\alpha^{-(n-1)}(a_{n+1}))(\cdots\\
&&(d(\alpha^{-(n-1)}(a_{n+k-1}))d(\alpha^{-(n-1)}(a_{n+k})))\cdots))))\cdots)\\
&+&d(\alpha(a_0))(\cdots(d(\alpha(a_{n-1}))(d(\alpha^{-1}(a_n)\alpha^{-n}(a_{n+1}))(d(\alpha^{-(n-1)}(a_{n+2}))(\cdots\\
&&(d(\alpha^{-(n-1)}(a_{n+k-1}))d(\alpha^{-(n-1)}(a_{n+k})))\cdots))))\cdots)\\
&=&(-1)^{n}d(a_0\alpha(a_1))(d(\alpha^{2}(a_{2}))(\cdots(d(\alpha^{2}(a_{n-1}))(d(\alpha(a_n))(d(\alpha^{-(n-1)}(a_{n+1}))(\cdots\\
&&(d(\alpha^{-(n-1)}(a_{n+k-1}))d(\alpha^{-(n-1)}(a_{n+k})))\cdots))))\cdots))\\
&+&\sum_{i=1}^{n-3}(-1)^{n-i}d(\alpha(a_0))(\cdots(d(\alpha(a_{i-1}))(d(a_{i}\alpha(a_{i+1}))(d(\alpha^{2}(a_{i+2}))(\cdots(d(\alpha^{2}(a_{n-1}))\\
&&(d(\alpha(a_n))(d(\alpha^{-(n-1)}(a_{n+1}))(\cdots(d(\alpha^{-(n-1)}(a_{n+k-1}))d(\alpha^{-(n-1)}(a_{n+k})))\cdots))))\cdots))))\cdots)\\
&+&d(\alpha(a_0))(\cdots(d(\alpha(a_{n-3}))(d(a_{n-2}\alpha(a_{n-1}))(d(\alpha(a_n))(d(\alpha^{-(n-1)}(a_{n+1}))(\cdots\\
&&(d(\alpha^{-(n-1)}(a_{n+k-1}))d(\alpha^{-(n-1)}(a_{n+k})))\cdots)))))\cdots)\\
&-&d(\alpha(a_0))(\cdots(d(\alpha(a_{n-2}))(d(a_{n-1}a_n)(d(\alpha^{-(n-1)}(a_{n+1}))(\cdots\\
&&(d(\alpha^{-(n-1)}(a_{n+k-1}))d(\alpha^{-(n-1)}(a_{n+k})))\cdots))))\cdots)\\
&+&d(\alpha(a_0))(\cdots(d(\alpha(a_{n-1}))(d(\alpha^{-1}(a_n)\alpha^{-n}(a_{n+1}))(d(\alpha^{-(n-1)}(a_{n+2}))(\cdots\\
&&(d(\alpha^{-(n-1)}(a_{n+k-1}))d(\alpha^{-(n-1)}(a_{n+k})))\cdots))))\cdots)\\
&=&d(\omega_n\omega_{k-1}),
\end{eqnarray*}
which is the graded Leibniz rule. Next, we verify by induction that the following identity holds:

\begin{equation}\label{derivative-of-monomials}d(a_0(da_1(\cdots(da_{n-1}da_n)\cdots)))=da_0(da_1(\cdots(da_{n-1}da_n)\cdots))\end{equation}

using the graded Leibniz rule and the equation $d^2=d\circ d=0$. For $a_0da_1\in \Omega^1(A)$,
$$d(a_0da_1)=da_0da_1+(-1)^0a_0d(d(a_1))=da_0da_1.$$
Suppose now that the identity
$$d(a_0(da_1(\cdots(da_{n-2}da_{n-1})\cdots)))=da_0(da_1(\cdots(da_{n-2}da_{n-1})\cdots))$$
holds for $a_0(da_1(\cdots(da_{n-2}da_{n-1})\cdots))\in\Omega^{n-1}(A)$, that is, if we replace $a_i$ with $a_{i+1}$ for $i=0,\cdots, n-1$, we have $d(a_1(da_2(\cdots(da_{n-1}da_{n})\cdots)))=da_1(da_2(\cdots(da_{n-1}da_{n})\cdots))$. Thus, for $a_0(da_1(\cdots(da_{n-1}da_{n})\cdots))\in\Omega^{n}(A)$,

\begin{eqnarray*}&&d(a_0(da_1(\cdots(da_{n-1}da_n)\cdots)))\\
&=&da_0(da_1(da_2(\cdots(da_{n-1}da_n)\cdots)))+(-1)^0a_0d(da_1(da_2(\cdots(da_{n-1}da_n)\cdots)))\\
&=&da_0(da_1(da_2(\cdots(da_{n-1}da_n)\cdots)))+(-1)^0a_0d(d(a_1(da_2(\cdots(da_{n-1}da_n)\cdots))))\\
&=&da_0(da_1(da_2(\cdots(da_{n-1}da_n)\cdots))).
\end{eqnarray*}
Let $(\Gamma,\gamma)$ be another Hom-DC on $(A,\alpha)$ with differential $\tilde{d}$ and let the morphism $\psi:\Omega(A)\to \Gamma$, in $\widetilde{\mathcal{H}}(\mathcal{M}_k)$, be given by
$$\psi(a)=a\:\mbox{and}\: \psi(a_0(da_1(\cdots(da_{n-1}da_n)\cdots)))=a_0(\tilde{d}a_1(\cdots(\tilde{d}a_{n-1}\tilde{d}a_n)\cdots)),\: n\geq 1$$
for $a\in A$, $a_0(da_1(\cdots(da_{n-1}da_n)\cdots))\in \Omega^n(A)$. Clearly, $\psi$ is surjective by its definition. Now, let $\mathcal{N}:=ker \psi$ be the kernel of $\psi$. From the equations (\ref{Hom-multiplication-of-monomials}) and (\ref{derivative-of-monomials}) it is concluded that $\mathcal{N}$ is a differential Hom-ideal of $\Omega(A)$. Thus, $\Gamma$ is identified with $\Omega(A)/\mathcal{N}$ showing the universality of $\Omega(A)$.

\section{Left-Covariant FODC over Monoidal Hom-Hopf Algebras}

\subsection{Left-Covariant Hom-FODC and Their Right Hom-ideals}

Let $(H,\alpha)$ be a monoidal Hom-Hopf algebra with a bijective antipode throughout the section. $(H,\alpha)$ is a left Hom-quantum space for itself with respect to the Hom-comultiplication $\Delta:H\to H\otimes H, h\mapsto h_1\otimes h_2$. Thus, by applying Definition \ref{left-covariant-Hom-FODC} to the monoidal Hom-Hopf algebra $(H,\alpha)$ we obtain the following
\begin{definition}\label{left-covariant-Hom-FODC1}A FODC $(\Gamma,\gamma)$ over the monoidal Hom-Hopf algebra $(H,\alpha)$ is said to be {\it left-covariant} if $(\Gamma,\gamma)$ is a left-covariant FODC over the left Hom-quantum space $(H,\alpha)$ with left Hom-coaction $\varphi=\Delta$ in Definition \ref{left-covariant-Hom-FODC}.
\end{definition}

\begin{remark}\label{left-covariant-Hom-FODC2}According to Proposition \ref{left-covariance-of-a-Hom-FODC}, an $(H,\alpha)$-Hom-FODC $(\Gamma,\gamma)$ is left-covariant if and only if there exists a morphism $\phi:\Gamma\to H\otimes \Gamma $ in $\widetilde{\mathcal{H}}(\mathcal{M}_k)$ such that, for $h,g\in H$,

\begin{equation}\label{left-coaction-on-Hom-FODC}\phi(h\cdot dg)=\Delta(h)(id\otimes d)(\Delta(g)).\end{equation}

In the proof of Proposition \ref{left-covariance-of-a-Hom-FODC}, it has been shown that if there is such a morphism $\phi$, it defines a left Hom-comodule structure of $(\Gamma,\gamma)$ on $(H,\alpha)$ and satisfies
$$\phi(\alpha(h)\cdot(\omega\cdot g))=\Delta(\alpha(h))(\phi(\omega)\Delta(g))$$

for $ h,g\in H$ and $\omega \in \Gamma$. From this it follows that $(\Gamma,\gamma)$ is a left-covariant $(H,\alpha)$-Hom-bimodule.
\end{remark}

Let $(\Gamma,\gamma)$ be a left-covariant $(H,\alpha)$-Hom-FODC with derivation $d:H\to \Gamma$. By the above remark $(\Gamma,\gamma)$ is a left-covariant $(H,\alpha)$-Hom-bimodule, and then by adapting the structure theory of left-covariant Hom-bimodules, which is discussed in Lemma 4.7 and Proposition 4.9 in \cite{Karacuha1}, to $(\Gamma,\gamma)$ we summarize the following results. We have the unique projection $P_L:(\Gamma,\gamma)\to (^{coH}\Gamma,\gamma|_{^{coH}\Gamma})$ given by $P_L(\varrho)=S(\varrho_{(-1)})\varrho_{(0)}$, for all $\varrho\in \Gamma$, such that
$$P_L(h\cdot\varrho)=\varepsilon(h)\gamma(P_L(\varrho)),\qquad \varrho=\varrho_{(-1)}P_L(\varrho_{(0)})$$ and
$$P_L(\varrho\cdot h)=\widetilde{ad}_R(h)(P_L(\varrho))=:P_L(\varrho)\lhd h$$
for any $h\in H$ and $\varrho\in \Gamma$. Let us now define a linear mapping $\omega_{\Gamma}:H\overset{d}{\rightarrow}\Gamma \overset{P_L}{\rightarrow} \:^{coH}\Gamma$ by
$$\omega_{\Gamma}(h)=P_L(dh),\: \forall h\in H.$$
Obviously, it is in $\widetilde{\mathcal{H}}(\mathcal{M}_k)$, that is, $\omega_{\Gamma}\circ \alpha=\gamma\circ\omega_{\Gamma}$. Since $\phi(dh)=(dh)_{(-1)}\otimes (dh)_{(0)}=(id \otimes d)(\Delta(h))=h_1\otimes dh_2$ by the above remark, we obtain

\begin{equation}\label{omega-for-Hom-FODC}\omega_{\Gamma}(h)=P_L(dh)=S(h_1)\cdot dh_2,\: \forall h\in H.\end{equation}

On the other hand, we can write $dh=(dh)_{(-1)}\cdot P_L((dh)_{(0)})=h_1\cdot P_L(dh_2)$, that is,

\begin{equation}\label{derivation-by-omega}dh=h_1\cdot\omega_{\Gamma}(h_2),\: \forall h\in H. \end{equation}

We will drop the subscript $\Gamma$ from $\omega_{\Gamma}(\cdot)$. By definition, for any $h\in H$, $\omega(h)\in\: ^{coH}\Gamma$. Conversely, let $\varrho=\sum_ih_i\cdot dg_i \in \: ^{coH}\Gamma$ for $h_i,g_i \in H$. Then

$$\varrho=P_L(\varrho)=\sum_i\varepsilon(h_i)\gamma(P_L(dg_i))=\sum_i\varepsilon(h_i)\gamma(\omega(g_i))=\sum_i\varepsilon(h_i)\omega(\alpha(g_i)),$$

showing that $\rho \in \omega(H)$. Thus, we get $\omega(H)=\: ^{coH}\Gamma$ which implies that $\Gamma=H\cdot\omega(H)=\omega(H)\cdot H$ and hence any $k$-linear basis of $\omega(H)$ is a left $(H,\alpha)$-Hom-module basis and a right $(H,\alpha)$-Hom-module basis for $(\Gamma,\gamma)$.

For $h,g\in H$, we get
\begin{eqnarray*}\omega(h)\lhd g &=&P_L(\omega(h)\cdot g)=P_L((S(h_1)\cdot dh_2)\cdot g)\\
&=&P_L(S(\alpha(h_1))\cdot (dh_2\cdot \alpha^{-1}(g)))=\varepsilon(S(\alpha(h_1)))\gamma(P_L(dh_2\cdot \alpha^{-1}(g)))\\
&=&\varepsilon(h_1)\gamma(P_L(dh_2\cdot \alpha^{-1}(g)))=\gamma(P_L(d(\alpha^{-1}(h))\cdot \alpha^{-1}(g)))\\
&=&\gamma(P_L(d(\alpha^{-1}(hg))-\alpha^{-1}(h)\cdot d(\alpha^{-1}(g))))\\
&=&\gamma(\omega(\alpha^{-1}(hg)))-\gamma(\varepsilon(\alpha^{-1}(h))\gamma(P_L(d(\alpha^{-1}(g)))))\\
&=&\omega(hg)-\varepsilon(h)\gamma^{2}(\omega(\alpha^{-1}(g)))=\omega(hg)-\varepsilon(h)\omega(\alpha(g))\\
&=&\omega(hg-\varepsilon(h)\alpha(g))=\omega((h-\varepsilon(h)1)g).
\end{eqnarray*}
Thus, by setting the notation $\bar{h}:=h-\varepsilon(h)1$, we have
\begin{equation}\omega(h)\lhd g=\widetilde{ad}_R(g)(\omega(h))=\omega(\bar{h}g),\end{equation}
and we rewrite the $(H,\alpha)$-Hom-bimodule structure as
\begin{equation}\label{Hom-bimod-structure1}g'\cdot(g\cdot\omega(h))=(\alpha^{-1}(g')g)\cdot\omega(\alpha(h)),\end{equation}
\begin{equation}\label{Hom-bimod-structure2} (g'\cdot\omega(h))\cdot g=(g'g_1)\cdot(\omega(h)\lhd g_2)=(g'g_1)\cdot\omega(\bar{h}g_2),\end{equation}
for $g,g',h \in H$.

In the following example we introduce the universal FODC over monoidal Hom-Hopf-algebra $(H,\alpha)$.
\begin{example}We define $(\Omega^{1}(H),\beta):=(H\otimes ker\varepsilon, \alpha\otimes\alpha')$, where $\alpha'=\alpha|_{ker\varepsilon}$. Let us denote the element $1\otimes \alpha^{-1}(\bar{g})=1\otimes \overline{\alpha^{-1}(g)}$, for $g\in H$, by $\omega(g)$. Thus we identify $g\otimes\bar{h}\in \Omega^1(H)$, where $g,h\in H$, with $g\cdot\omega(h)$. We then introduce the Hom-bimodule structure of $\Omega^{1}(H)$ as in (\ref{Hom-bimod-structure1}) and (\ref{Hom-bimod-structure2}), for all $g,g',h\in H$,
$$g'\cdot(g\cdot\omega(h)):=(\alpha^{-1}(g')g)\cdot\omega(\alpha(h)),$$
$$(g'\cdot\omega(h))\cdot g:=(g'g_1)\cdot\omega(\bar{h}g_2),$$
and a linear mapping
$$d:H\to \Omega^1(H),\: h\mapsto h_1\otimes \overline{h_2}=h_1\cdot\omega(h_2).$$
For any $g,h \in H$,

\begin{eqnarray*}g\cdot dh+dg \cdot h &=&g\cdot(h_1\cdot \omega(h_2))+(g_1\cdot\omega(g_2))\cdot h\\
&=&(\alpha^{-1}(g)h_1)\cdot \omega(\alpha(h_2))+(g_1h_1)\cdot\omega(\overline{g_2}h_2)\\
&=&(\alpha^{-1}(g)h_1)\cdot \omega(\alpha(h_2))+(g_1h_1)\cdot\omega(g_2h_2)-(g_1h_1)\omega((\varepsilon(g_2)1)h_2)\\
&=&(\alpha^{-1}(g)h_1)\cdot \omega(\alpha(h_2))+(g_1h_1)\cdot\omega(g_2h_2)-(\alpha^{-1}(g)h_1)\cdot \omega(\alpha(h_2))\\
&=&(g_1h_1)\cdot\omega(g_2h_2)=(gh)_1\omega((gh)_2)\\
&=&d(gh),
\end{eqnarray*}
showing that $d$ satisfies the Leibniz rule.
$$d(\alpha(h))=\alpha(h_1)\cdot\omega(\alpha(h_2))=\alpha(h_1)\cdot\beta(\omega(h_2))=\beta(h_1\cdot\omega(h_2))=\beta(dh),$$
which means that $d\in \widetilde{\mathcal{H}}(\mathcal{M}_k)$.
\begin{eqnarray*}\omega(h)&=&\omega(\alpha(\varepsilon(h_1)h_2))=\varepsilon(h_1)\omega(\alpha(h_2))\\
&=&\varepsilon(h_1)\beta(\omega(h_2))=(\varepsilon(h_1)1)\cdot\omega(h_2)\\
&=&(S(h_{11})h_{12})\cdot \omega(h_2)=(S(\alpha^{-1}(h_1))h_{21})\cdot \omega(\alpha(h_{22}))\\
&=&\alpha(S(\alpha^{-1}(h_1)))\cdot(h_{21}\cdot \beta^{-1}(\omega(\alpha(h_{22}))))\\
&=&S(h_1)\cdot(h_{21}\cdot\omega(h_{22}))=S(h_1)\cdot d(h_2),
\end{eqnarray*}
which proves that $\Omega^1(H)=H\cdot dH$. Therefore, $(\Omega^1(H),\beta)$ is an $(H,\alpha)$-Hom-FODC.

For another $(H,\alpha)$-Hom-FODC $(\Gamma,\gamma)$ with differentiation $\bar{d}:H\to \Gamma$, let us define the linear map $\psi: \Omega^1(H)\to \Gamma$ by $\psi(h\cdot dg)=h\cdot \bar{d}g$, where $g,h\in H$. It is well-defined: Suppose that $\sum_ih_i\cdot dg_i=0$ in $\Omega^1(H)$, where $h_i,g_i\in H$. Then we have
\begin{eqnarray*}\sum_ih_i\cdot dg_i&=&\sum_ih_i\cdot(g_{i,1}\otimes\overline{g_{i,2}})=\sum_i\alpha^{-1}(h_i)g_{i,1}\otimes\alpha(\overline{g_{i,2}})\\
&=&\sum_i\alpha^{-1}(h_i)g_{i,1}\otimes\overline{\alpha(g_{i,2})}=\sum_i\alpha^{-1}(h_i)g_{i,1}\otimes(\alpha(g_{i,2})-\varepsilon(\alpha(g_{i,2}))1)\\
&=&\sum_i[\alpha^{-1}(h_i)g_{i,1}\otimes\alpha(g_{i,2})-\alpha^{-1}(h_i)g_{i,1}\varepsilon(g_{i,2})\otimes 1)]\\
&=&\sum_i[\alpha^{-1}(h_i)g_{i,1}\otimes\alpha(g_{i,2})-\alpha^{-1}(h_ig_i)\otimes 1)]=0.
\end{eqnarray*}
So, by applying $(m\otimes id)\circ \tilde{a}^{-1}\circ(id\otimes S\otimes id)\circ(id\otimes \Delta)$ to

$$\sum_ih_i\cdot dg_i=\sum_i[\alpha^{-1}(h_i)g_{i,1}\otimes\alpha(g_{i,2})-\alpha^{-1}(h_ig_i)\otimes 1)]=0,$$

we acquire the equality $\sum_i(h_i\otimes g_i-\alpha^{-1}(h_ig_i)\otimes 1)=0$. Thus $\sum_ih_i\cdot \bar{d}g_i=0$ in $\Gamma$ concluding that $\psi$ is well-defined. On the other hand we prove that $\psi\in \widetilde{\mathcal{H}}(\mathcal{M}_k)$:

\begin{eqnarray*}\psi(\beta(h\cdot dg))&=&\psi(\alpha(h)\cdot \beta(dg))=\psi(\alpha(h)\cdot d(\alpha(g)))\\
&=&\alpha(h)\cdot \bar{d}(\alpha(g))=\alpha(h)\cdot \gamma(\bar{d}(g))=\gamma(h\cdot \bar{d}g)=\gamma(\psi(h\cdot dg)).
\end{eqnarray*}

The subobject $(ker\psi, \beta|_{ker\psi})=(\mathcal{N},\beta')$ is an $(H,\alpha)$-Hom-subbimodule of $(\Omega^1(H),\beta)$: Indeed, for $h'\in H$ and $h\cdot dg \in \mathcal{N} $,

\begin{eqnarray*}\psi(h'\cdot(h\cdot dg))&=&\psi((\alpha^{-1}(h')h)\cdot d(\alpha(g)))=(\alpha^{-1}(h')h)\cdot \bar{d}(\alpha(g))\\
&=&h'\cdot(h\cdot \bar{d}g)=h'\cdot \psi(h\cdot dg)=0,
\end{eqnarray*}

\begin{eqnarray*}\psi((h\cdot dg)\cdot h')&=&\psi(\alpha(h)\cdot d(g\alpha^{-1}(h'))-(hg)\cdot dh')\\
&=&\alpha(h)\cdot \bar{d}(g\alpha^{-1}(h'))-(hg)\cdot \bar{d}h'=\alpha(h)\cdot \bar{d}(g\alpha^{-1}(h'))-\alpha(h)\cdot(g\cdot\bar{d}(\alpha^{-1}(h')))\\
&=&\alpha(h)\cdot(\bar{d}g\cdot \alpha^{-1}(h'))=(h\cdot \bar{d}g)\cdot h'=0.
\end{eqnarray*}
Hence we have the quotient object $(\Omega^1(H)/\mathcal{N},\bar{\beta})$ as $(H,\alpha)$-Hom-bimodule, where the automorphism $\bar{\beta}$ is induced by $\beta$ and define the $(H,\alpha)$-bilinear map $\bar{\psi}:\Omega^1(H)/\mathcal{N}\to \Gamma,\: h\cdot dg\mapsto h\cdot\bar{d}g$, which is surjective by definition. Since $ker\bar{\psi}=\mathcal{N}$, $\bar{\psi}$ is 1-1, showing that $\Gamma$ is isomorphic to the quotient $\Omega^1(H)/\mathcal{N}$. Therefore $(\Omega^1(H),\beta)$ is the universal Hom-FODC over $(H,\alpha)$.
\end{example}

We define the subobject

\begin{equation}\mathcal{R}_{\Gamma}=\{h\in ker\varepsilon|\:\omega_{\Gamma}(h)=0\}\end{equation}

of $(ker\varepsilon,\alpha|_{ker\varepsilon})$ for a given left-covariant $(H,\alpha)$-Hom-FODC $(\Gamma,\gamma)$, which is clearly a Hom-ideal of $(H,\alpha)$.
We now prove that there is a one-to-one correspondence between left-covariant $(H,\alpha)$-Hom-FODC's and right Hom-ideals $\mathcal{R}$.

\begin{proposition}\label{one-one-corresp-Hom-ideals}
\begin{enumerate}
\item Let $(\mathcal{R},\alpha'')$ be a right Hom-ideal of $(H,\alpha)$ which is a subobject of $(ker\varepsilon,\alpha')$, where $\alpha''=\alpha|_{\mathcal{R}}$. Then $\mathcal{N}:=H\cdot \omega_{\Omega^1(H)}(\mathcal{R})$ is an $(H,\alpha)$-Hom-subbimodule of $(\Omega^1(H),\beta)$. Furthermore, $(\Gamma,\gamma):=(\Omega^1(H)/\mathcal{N},\bar{\beta})$ is a left-covariant Hom-FODC over $(H,\alpha)$ such that $\mathcal{R}_{\Gamma}=\mathcal{R}$.
\item  For a given left-covariant $(H,\alpha)$-Hom-FODC $(\Gamma,\gamma)$, $\mathcal{R}_{\Gamma}$ is a right Hom-ideal of $(H,\alpha)$ and $\Gamma$ is isomorphic to $\Omega^1(H)/H\cdot \omega_{\Omega^1(H)}(\mathcal{R}_{\Gamma})$.
\end{enumerate}
\end{proposition}

\begin{proof}\begin{enumerate}
\item For any $h\in \mathcal{R}$ and $g\in H$, we have
\begin{eqnarray*}\omega(h)\cdot g&=&(1g_1)\cdot (\beta^{-1}(\omega(h))\lhd g_2)=\alpha(g_1)\cdot(\omega(\alpha^{-1}(h))\lhd g_2)\\
&=&\alpha(g_1)\cdot\omega(\alpha^{-1}(\bar{h})g_2)=\alpha(g_1)\cdot\omega(\alpha^{-1}(h)g_2),
\end{eqnarray*}
which is in $H\cdot \omega_{\Omega^1(H)}(\mathcal{R})$, and hence $\mathcal{N}=H\cdot \omega_{\Omega^1(H)}(\mathcal{R})$ is an $(H,\alpha)$-Hom-subbimodule of $\Omega^1(H)=H\cdot \omega_{\Omega^1(H)}(H)$. So, $(\Gamma=\Omega^1(H)/\mathcal{N},\bar{\beta})$ is a $(H,\alpha)$-Hom-FODC with differentiation $\bar{d}:H\to \Gamma, \: h\mapsto \bar{d}h=\pi(dh)=h_1\cdot \omega(h_2)+\mathcal{N}$, where $\pi:\Omega^1(H)\to \Omega^1(H)/\mathcal{N}$ is the natural projection.

Let $\phi:\Omega^1(H)\to H\otimes \Omega^1(H),\: h\cdot \omega(g)\mapsto \alpha(h_1)\otimes h_2\cdot \omega(\alpha^{-1}(g))$ be the Hom-coaction for the left-covariant Hom-FODC $(\Omega^1(H),\beta)$. Since, for $h\cdot \omega(r)\in \mathcal{N}$ we have
$$\phi(h\cdot \omega(r))=\alpha(h_1)\otimes h_2\cdot \omega(\alpha^{-1}(r))\: \in  H\otimes\mathcal{N},$$
that is, $\phi(\mathcal{N})\subseteq H\otimes\mathcal{N}$, $\phi$ passes to a left Hom-action of $(H,\alpha)$ on $(\Gamma,\bar{\beta})$ as $\bar{\phi}(h\cdot\omega(g)+\mathcal{N})=\alpha(h_1)\otimes( h_2\cdot \omega(\alpha^{-1}(g))+\mathcal{N})$. For $g,h\in H$, we get

\begin{eqnarray*}\lefteqn{\Delta(g)(id \otimes \bar{d})(\Delta(h))}\hspace{4em}\\
&=&(g_1\otimes g_2)(h_1\otimes \bar{d}h_2)\\
&=&g_1h_1\otimes g_2\cdot(h_{21}\cdot\omega(h_{22})+\mathcal{N})=g_1h_1\otimes( g_2\cdot(h_{21}\cdot\omega(h_{22}))+\mathcal{N})\\
&=&g_1h_1\otimes( (\alpha^{-1}(g_2)h_{21})\cdot\omega(\alpha(h_{22}))+\mathcal{N})\\
&=&g_1\alpha(h_{11})\otimes( (\alpha^{-1}(g_2)h_{12})\cdot\omega(h_2)+\mathcal{N})=\bar{\phi}(\alpha^{-1}(g)h_{1}\cdot\omega(\alpha(h_2))+\mathcal{N})\\
&=&\bar{\phi}(g\cdot(h_{1}\cdot\omega(h_2))+\mathcal{N})=\bar{\phi}(g\cdot\bar{d}h),
\end{eqnarray*}
proving the left-covariance of $(\Gamma,\bar{\beta})$ with respect to $(H,\alpha)$. Thus, we have the projection $\overline{P}_L:\Gamma\to\: ^{coH}\Gamma$ given by
$$\overline{P}_L(h\cdot\omega(g)+\mathcal{N})=\varepsilon(h)\omega(\alpha(g))+\mathcal{N}$$
for $h\cdot\omega(g)\in \Omega^1(H)$.

For $h\in \mathcal{R}$, $$\omega_{\Gamma}(h)=\overline{P}_L(\bar{d}h)=\overline{P}_L(h_1\cdot\omega(h_2)+\mathcal{N})=\varepsilon(h_1)\omega(\alpha(h_2))+\mathcal{N}=\omega(h)+\mathcal{N}
=\mathcal{N}=0_{\Gamma},$$
implying that $\mathcal{R}\subseteq \mathcal{R}_{\Gamma}$. On the contrary, if $\omega_{\Gamma}(h)=0_{\Gamma}$ for some $h\in ker\varepsilon$, then $\omega(h)\in \mathcal{N}=H\cdot \omega(\mathcal{R})$, that is, $h\in \mathcal{R}$, i.e., $\mathcal{R}_{\Gamma}\subseteq \mathcal{R}$. Therefore, $\mathcal{R}=\mathcal{R}_{\Gamma}$.

\item Since $(\Gamma,\gamma)$ is a left-covariant Hom-FODC, $\widetilde{ad}_R(g)(\omega(h))=\omega(\bar{h}g)$ holds for $g,h\in H$. Hence, for $h\in \mathcal{R}_{\Gamma}$ and $g\in ker\varepsilon$, we have $\omega_{\Gamma}(hg)=\omega_{\Gamma}(\bar{h}g)=\widetilde{ad}_R(g)(\omega_{\Gamma}(h))=0$ since $\omega_{\Gamma}(h)=0$. Therefore, $\mathcal{R}_{\Gamma}$ is a subobject of $ker\varepsilon$ which is a right Hom-ideal of $(H,\alpha)$. Thus, $\Gamma\simeq \Omega^1(H)/H\cdot \omega_{\Omega^1(H)}(\mathcal{R}_{\Gamma})$ by (1).
\end{enumerate}
\end{proof}

\subsection{Quantum Hom-Tangent Space}
In the theory of Lie groups, if $A=C^{\infty}(G)$ is the algebra of smooth functions on a Lie group $G$ and $\mathcal{R}$ is the ideal of $A$ consisting of all functions vanishing with first derivatives at the neutral element of $G$, then the vector space of all linear functionals on $A$ annihilating $1$ of $A$ and $\mathcal{R}$ is identified with the tangent space at the neutral element, i.e., with the Lie algebra of $G$. In the theory of quantum groups, this consideration gives rise to the notion of quantum tangent space associated to a left-covariant FODC $\Gamma$ on a Hopf algebra $A$, which is defined as the vector space

$$\mathcal{T}_{\Gamma}=\{X\in A'|\:X(1)=0,\: X(a)=0,\forall a\in \mathcal{R}_{\Gamma}\}, $$

where $\mathcal{R}_{\Gamma}=\{a\in ker\varepsilon_A|\:P_L(da)=0\}$. In what follows, we study the Hom-version of the quantum tangent space.

We recall that the dual monoidal Hom-algebra $(H',\bar{\alpha})$ of $(H,\alpha)$ consists of functionals $X:H\to k$ and is equipped with the convolution product  $(XY)(h)=X(h_1)Y(h_2)$, for $X,Y\in H'$ and $h\in H$, as Hom-multiplication and with the Hom-unit $\varepsilon:H\to k,$ where automorphism $\bar{\alpha}:H'\to H'$ is given by $\bar{\alpha}(X)=X\circ \alpha^{-1}$. The morphism
$$H'\otimes H\to H, X\otimes h\mapsto X\bullet h:=\alpha^{2}(h_1)X(\alpha(h_2)),$$
in $\widetilde{\mathcal{H}}(\mathcal{M}_k)$, makes $(H,\alpha)$ a left $(H',\bar{\alpha})$-Hom-module.

\begin{definition}Let $(\Gamma, \gamma)$ be a left-covariant $(H,\alpha)$-Hom-FODC. Then the subobject
\begin{equation}\mathcal{T}_{\Gamma}=\{X\in H'|\:X(1)=0,\: X(h)=0,\forall h\in \mathcal{R}_{\Gamma}\} \end{equation}
of $(H',\bar{\alpha})$, in $\widetilde{\mathcal{H}}(\mathcal{M}_k)$, is said to be the {\it quantum Hom-tangent space} to $(\Gamma, \gamma)$.
\end{definition}

\begin{proposition}\label{bilinear-form}Let $(\Gamma, \gamma)$ be a left-covariant $(H,\alpha)$-Hom-FODC and $(\mathcal{T}_{\Gamma},\bar{\alpha}')$ be the quantum Hom-tangent space to it, where $\bar{\alpha}'=\bar{\alpha}|_{\mathcal{T}_{\Gamma}}$. Then, there is a unique bilinear form $<\cdot,\cdot>:\mathcal{T}_{\Gamma} \times \Gamma \to k$ in $\widetilde{\mathcal{H}}(\mathcal{M}_k)$ such that

  \begin{equation} <X,h\cdot dg>=\varepsilon(h)X(g),\:\forall g,h\in H,\:X\in \mathcal{T}_{\Gamma}.\end{equation}

With respect to this bilinear form, $(\mathcal{T}_{\Gamma},\bar{\alpha}')$ and $(\:^{coH}\Gamma,\gamma')=(\omega(H),\gamma')$ form a nondegenerate dual pairing, where $\gamma'=\gamma|_{^{coH}\Gamma}$. Moreover, we have
 \begin{equation}\label{bil-form-wrt-omega} <X,\omega(h)>=X(\alpha^{-1}(h)), \forall h\in H, X\in \mathcal{T}_{\Gamma}.\end{equation}
\end{proposition}

\begin{proof} We define $<X,\varrho>:=X(\sum_i \varepsilon(h_i)g_i)=\sum_i \varepsilon(h_i)X(g_i)$ for $X\in H'$ and $\varrho=\sum_ih_i\cdot dg_i \in \Gamma$. Suppose that $\varrho=\sum_ih_i\cdot dg_i=0$. Then
\begin{eqnarray*}0&=&P_L(\gamma^{-1}(\rho))=\sum_iP_L(\alpha^{-1}(h_i)d(\alpha^{-1}(g_i)))\\
&=&\sum_i\varepsilon(\alpha^{-1}(h_i))\gamma(P_L(d(\alpha^{-1}(g_i))))\\
&=&\sum_i\varepsilon(h_i)\gamma(\omega(\alpha^{-1}(g_i)))\\
&=&\omega\left(\sum_i\varepsilon(h_i)g_i\right),
\end{eqnarray*}
hence $\omega(\sum_i\varepsilon(h_i)\overline{g_i})=0$, which implies that $\sum_i\varepsilon(h_i)\overline{g_i}\in \mathcal{R}_{\Gamma}$. Thus, by the definition of $\mathcal{T}_{\Gamma}$ we get
\begin{eqnarray*}<X,\rho>&=&X\left(\sum_i \varepsilon(h_i)g_i\right)=X\left(\sum_i (\varepsilon(h_i)\overline{g_i}+\varepsilon(h_i)\varepsilon(g_i)1)\right)\\
&=&X\left(\sum_i \varepsilon(h_i)\overline{g_i}\right)+\sum_i\varepsilon(h_i)\varepsilon(g_i)X(1)=0,
\end{eqnarray*}
which proves that the bilinear form $<\cdot,\cdot>$ is well-defined. Uniqueness comes immediately from the fact that $\Gamma=H\cdot dH$. Since

\begin{eqnarray*}<\bar{\alpha}(X),\gamma(\varrho)>&=&(X\circ \alpha^{-1})\left(\sum_i \varepsilon(\alpha(h_i))\alpha(g_i)\right)=\sum_i \varepsilon(\alpha(h_i))(X\circ \alpha^{-1})(\alpha(g_i))\\
&=&X\left(\sum_i \varepsilon(h_i)g_i\right)=<X,\varrho>,
\end{eqnarray*}
the bilinear form $<\cdot,\cdot>$ is in $\widetilde{\mathcal{H}}(\mathcal{M}_k)$.
For any $h\in H$, $<X,\omega(h)>=<X,S(h_1)\cdot dh_2>=\varepsilon(h_1)X(h_2)=X(\varepsilon(h_1)h_2)=X(\alpha^{-1}(h))$, which is the formula (\ref{bil-form-wrt-omega}). For any $h\in ker\varepsilon$, if $<X,\omega(h)>=X(\alpha^{-1}(h))=0,\: \forall X \in \mathcal{T}_{\Gamma}$, then $\alpha^{-1}(h)\in \mathcal{R}_{\Gamma}$: Suppose that the element $0\neq \alpha^{-1}(h)\in ker\varepsilon$ is not contained in $\mathcal{R}_{\Gamma}$. Then we can extend $\alpha^{-1}(h)$ to a basis of $ker\varepsilon$ and find a functional $X\in \mathcal{T}_{\Gamma}$ such that $X(\alpha^{-1}(h))\neq0$, which contradicts with the hypothesis of the statement. So we have $h\in \mathcal{R}_{\Gamma}$ since $\omega\circ\alpha^{-1}=\gamma^{-1}\circ\omega$. On the other hand $<X,\omega(h)>=X(\alpha^{-1}(h))=\bar{\alpha}(X)(h)=0$ for all $\omega(h)\in \omega(H)$ implies $\bar{\alpha}(X)=0$, that is, $X=0$. Hence, $(\mathcal{T}_{\Gamma},\bar{\alpha}')$ and $(\:^{coH}\Gamma,\gamma')=(\omega(H),\gamma')$ form a nondegenerate dual pairing with respect to $<\cdot,\cdot >$.
\end{proof}

Let $\{X_i\}_{i\in I}$ be a linear basis of $\mathcal{T}_{\Gamma}$ and $\{\omega_i\}_{i\in I}$ be the dual basis of $^{coH}\Gamma$, that is, $<X_i,\omega_j>=\delta_{ij}$ for $i,j \in I$. Also, from Theorem (4.17) in \cite{Karacuha1}, recall the family of functionals $\{f^{i}_j\}_{i,j\in I}$ in the definition of the Hom-action $^{coH}\Gamma\otimes H\to  \:^{coH}\Gamma, \omega_i\otimes h \mapsto \omega_i\lhd h=f^{i}_j(h)\omega_j$, where all but finitely many $f^{i}_j(h)$ vanish and Einstein summation convention is used. These functionals satisfy, for all $h,g\in H$ and $i,j\in I$,
$$f^{i}_j(hg)=(\bar{\gamma}^{i}_kf^k_l)(h)f^l_j(\alpha(g)),\:\: f^{i}_j(1)=\gamma^{i}_j,$$
where $\gamma'(\omega_i)=\gamma^{i}_j\omega_j$ and $\gamma'^{-1}(\omega_{i})=\bar{\gamma}^{i}_j\omega_j$ such that $\gamma^{i}_j\bar{\gamma}^{j}_k=\delta_{ik}=\bar{\gamma}^{i}_j\gamma^{j}_k$.

\begin{proposition}For $h,g\in H$, we have
\begin{equation}\label{derivation-by-basis-of-tangent-space}dh=(X_i\bullet \alpha^{-2}(h))\cdot\omega_i,\end{equation}
\begin{equation}\label{action-of-tangents-on-product}X_i(hg)=\varepsilon(h)(\gamma_i^{j}X_j)(g)+X_k(h)(\gamma_i^l\bar{f}_l^k)(g),\end{equation}
where $\bar{f}^k_l=\bar{\gamma}^k_pf^p_l$.
\end{proposition}

\begin{proof} By the formula \ref{bil-form-wrt-omega}, we have $<X_i,\omega(h)>=X_i(\alpha^{-1}(h))$ implying $\omega(h)=X_i(\alpha^{-1}(h))\omega_i$. Thus,
$dh=h_1\cdot\omega(h_2)=h_1\cdot(X_i(\alpha^{-1}(h_2)\omega_i))=(X_i\bullet \alpha^{-2}(h))\cdot\omega_i$ which is the formula \ref{derivation-by-basis-of-tangent-space}.
By using this formula and the Leibniz rule, we obtain
\begin{eqnarray*}(X_l\bullet \alpha^{-2}(hg))\cdot\omega_l&=&d(hg)=dh\cdot g+h\cdot dg\\
&=&((X_j\bullet \alpha^{-2}(h))\cdot\omega_j)\cdot g+h\cdot((X_i\bullet \alpha^{-2}(g))\cdot\omega_i)\\
&=&\alpha(X_j\bullet \alpha^{-2}(h))\cdot (\omega_j\cdot\alpha^{-1}(g))+(\alpha^{-1}(h) (X_i\bullet \alpha^{-2}(g)))\cdot\gamma'(\omega_i)\\
&=&\alpha(X_j\bullet \alpha^{-2}(h))\cdot((\bar{f}^{j}_k\bullet\alpha^{-2}(g))\cdot\omega_k)+(\alpha^{-1}(h)(X_i\bullet \alpha^{-2}(g)))\cdot(\gamma^{i}_k\omega_k)\\
&=&((X_j\bullet \alpha^{-2}(h))(\bar{f}^{j}_k\bullet\alpha^{-2}(g)))\cdot(\gamma^k_l\omega_l)+(\alpha^{-1}(h)(X_i\bullet \alpha^{-2}(g)))\cdot(\gamma^{i}_l\omega_l)\\
&=&[(X_j\bullet \alpha^{-2}(h))((\gamma^k_l\bar{f}^{j}_k)\bullet\alpha^{-2}(g))+\alpha^{-1}(h)((\gamma^{i}_lX_i)\bullet \alpha^{-2}(g))]\cdot\omega_l,
\end{eqnarray*}
hence, by replacing $\alpha^{-2}(h)$ and $\alpha^{-2}(g)$ by $h$ and $g$, respectively, we get
$$X_l\bullet(hg)=\alpha(h)((\gamma^{i}_lX_i)\bullet g)+(X_j\bullet h)((\gamma^k_l\bar{f}^{j}_k)\bullet g), l\in I.$$
By applying $\varepsilon$ to the both sides of this equation we acquire
 $$X_l(hg)=\varepsilon(h)(\gamma_i^lX_i)(g)+X_j(h)(\gamma_k^l\bar{f}_j^k)(g),$$
since, for any $h\in H$ and $f\in H'$, the equality $\varepsilon(f\bullet h)=\varepsilon(\alpha^{2}(h_1))f(\alpha(h_2))=\varepsilon(h_1)f(\alpha(h_2))=f(\alpha(\varepsilon(h_1)h_2))=f(h)$ holds.
\end{proof}

Let $(A,\alpha)$ be a monoidal Hom-algebra. Then we consider $A'\otimes A'$, where $A'=Hom(A,k)$, as a linear subspace of $(A\otimes A)'$ by identifying $f\otimes g\in A'\otimes A'$ with the linear functional on $A\otimes A$ specified by $(f\otimes g)(a\otimes a'):=f(a)g(a')$ for $a,a'\in A$. For $f\in H'$, let us define $\Delta(f)\in (A\otimes A)'$ by $\Delta(f)(a\otimes b):=f(ab)$ for $a,b\in A$. We now denote, by $A^{\circ}$, the set of all functionals $f\in A'$ such that $\Delta(f)\in A'\otimes A'$, i.e., it is written as a finite sum $$\Delta(f)=\sum_{p=1}^{P}f_p\otimes g_p$$ for some functionals $f_p,g_p\in A',\: p=1,...,P$, where $P$ is a natural number so that we have $f(ab)=\sum_pf_p(a)g_p(b)$. Then $(A^{\circ},\alpha^{\circ})$ is a monoidal Hom-coalgebra with Hom-comultiplication given above and the Hom-counit is defined by $\varepsilon(f)=f(1_A)$, where $\alpha^{\circ}(f)=f\circ\alpha^{-1}$ for any $f\in A^{\circ}$: Let $f\in A^{\circ}$ and $\Delta(f)=\sum_pf_p\otimes g_p$ such that the functionals $\{f_p\}_{p=1}^{P}$ are chosen to be linearly independent. So, one can find $a_q\in A$ such that $f_p(a_q)=\delta_{pq}$. Thus we get
\begin{eqnarray*}g_q(ab)&=&\sum_p\delta_{qp}g_p(ab)=\sum_pf_p(a_q)g_p(ab)=f(a_q(ab))\\
&=&f((\alpha^{-1}(a_q)a)\alpha(b))=\sum_pf_p(\alpha^{-1}(a_q)a)g_p(\alpha(b)),
\end{eqnarray*}
showing that $g_q\in A^{\circ}$, and analogously $f_q\in A^{\circ}$, and hence $\Delta(f)\in A^{\circ}\otimes A^{\circ}$. Let $f\in A^{\circ}$ and $a,b,c\in A$. Then we have the Hom-coassociativity of $\Delta$:
$$(\bar{\alpha}^{-1}\otimes \Delta)(\Delta(f))(a\otimes b \otimes c)=f(\alpha(a)(bc))=f((ab)\alpha(c))=(\Delta\otimes\bar{\alpha}^{-1})(\Delta(f))(a\otimes b \otimes c).$$
On the other hand, $$(id \otimes \varepsilon)(\Delta(f))(h)=\left(\sum_p\bar{\alpha}(f_p)g_p(1_A)\right)(h)=\sum_pf_p(\alpha^{-1}(h))g_p(1_A)=f(h)$$
shows that Hom-counity is satisfied.

Suppose that $(A,\alpha)$ is a monoidal Hom-bialgebra, then the monoidal Hom-coalgebra $(A^{\circ},\alpha^{\circ})$ endowed with the convolution product, as in the argument before Lemma (4.16) in \cite{Karacuha1}, is as well a monoidal Hom-bialgebra with the Hom-unit given by the Hom-counit $\varepsilon$ of the monoidal Hom-coalgebra $(A,\alpha)$: One can easily check the compatibility condition between Hom-comultiplication and Hom-multiplication of $(A^{\circ},\alpha^{\circ})$ which follows from that of $(A,\alpha)$. So, it suffices to verify that for any $f,g\in A^{\circ} $, $fg$ is also in $A^{\circ}$: If we put $\Delta(f)=\sum_pf_p\otimes g_p$ and $\Delta(g)=\sum_qh_q\otimes k_q$, then we get
$$(fg)(ab)=\Delta(fg)(a\otimes b)=\sum_{p,q}f_ph_q(a) g_pk_q(b)=\left(\sum_{p,q}f_ph_q\otimes g_pk_q\right)(a\otimes b),$$
so that $fg\in A^{\circ}$.

If $(A,\alpha)$ is a monoidal Hom-Hopf algebra, then so is $(A^{\circ},\alpha^{\circ})$ with antipode defined by $S(f)(a)=f(S(a))$ for $f\in A^{\circ}$ and $a\in A$: Set $\Delta(f)=\sum_pf_p\otimes g_p$, and then we obtain
$$\Delta(S(f))(a\otimes b)=S(f)(ab)=f(S(ab))=\sum_pS(f_p)(b) S(g_p)(a)=\left(\sum_p S(g_p)\otimes S(f_p)\right)(a\otimes b),$$
implying $S(f)\in A^{\circ}$. Lastly, for $a\in A$, we have
$$((m(S\otimes id)\Delta)(f))(a)=\sum_p(S(f_p)g_p)(a)=\varepsilon(a)f(1)=1_{A^{\circ}}(a)\varepsilon_{A^{\circ}}(f)=((\eta\circ\varepsilon)(f))(a),$$
similarly we get $((m(id\otimes S)\Delta)(f))(a)=((\eta\circ\varepsilon)(f))(a)$.

We then call the monoidal Hom-coalgebra (respectively, Hom-bialgebra, Hom-Hopf algebra) $A^{\circ}$ above the {\it dual monoidal Hom-coalgebra} (respectively, {\it Hom-bialgebra}, {\it Hom-Hopf algebra}). Suppose now that the vector space $\mathcal{T}_{\Gamma}$ is finite dimensional. Then we assert from Theorem (4.17) in \cite{Karacuha1} and (\ref{action-of-tangents-on-product}) that the functionals $f^{i}_j$ and $X_l$ are in the dual monoidal Hom-Hopf algebra $H^{\circ}$ and we have the following equations, where there is summation over repeating indices,

\begin{equation}\Delta(f^{i}_j)=\bar{f}^{i}_l\otimes f^l_j\circ\alpha,\end{equation}
\begin{equation}\Delta(X_l)=X_j\otimes\gamma^k_l\bar{f}^j_k +\varepsilon\otimes \gamma^l_{i}X_i\end{equation}
in $H^{\circ}$.

\section{Bicovariant FODC over Monoidal Hom-Hopf Algebras}

\subsection{Right-Covariant Hom-FODC }

\begin{definition}\label{righ-covariant-Hom-FODC}Let $(H,\beta)$ be a monoidal Hom-bialgebra. A FODC $(\Gamma,\gamma)$ over a right Hom-quantum space $(A,\alpha)$ with right Hom-coaction $\varphi:A\to A\otimes H,\: a\mapsto a_{[0]}\otimes a_{[1]}$ is called {\it right-covariant} with respect to $(H,\beta)$ if there exists a right Hom-coaction $\phi:\Gamma\to \Gamma\otimes H,\: \omega\mapsto \omega_{[0]}\otimes \omega_{[1]}$ of $(H,\beta)$ on $(\Gamma,\gamma)$ such that
\begin{enumerate}
\item $\phi(\alpha(a)\cdot(\omega\cdot b))=\varphi(\alpha(a))(\phi(\omega)\varphi(b))[=(\varphi(a)\phi(\omega))\varphi(\alpha(b))=\phi(a\cdot\omega)\cdot b]$, $\forall a,b \in A$, $\omega\in \Gamma$,
\item $\phi(da)=(d\otimes id)(\varphi(a)) $, $\forall a\in A$
\end{enumerate}
\end{definition}
Let $(H,\alpha)$ be a monoidal Hom-Hopf algebra with an invertible antipode $S$. Since $(H,\alpha)$ is a right Hom-quantum space for itself with respect to the Hom-comultiplication $\Delta:H\to H\otimes H, h\mapsto h_1\otimes h_2$, the above definition induces the following definition.

\begin{definition}\label{righ-covariant-Hom-FODC1}A $(H,\alpha)$-Hom-FODC $(\Gamma,\gamma)$ is said to be {\it right-covariant} if $(\Gamma,\gamma)$ is a right-covariant FODC over the right Hom-quantum space $(H,\alpha)$ with right Hom-coaction $\varphi=\Delta$ in the above definition, or in an equivalent way if there is a morphism $\phi:\Gamma\to \Gamma\otimes H $ in $\widetilde{\mathcal{H}}(\mathcal{M}_k)$ such that, for $h,g\in H$,
\begin{equation}\label{right-coaction-on-Hom-FODC}\phi(h\cdot dg)=\Delta(h)(d\otimes id)(\Delta(g)).\end{equation}
\end{definition}

If we modify the Proposition \ref{left-covariance-of-a-Hom-FODC} to the right-covariant case, we conclude that the right-covariant $(H,\alpha)$-Hom-FODC $(\Gamma,\gamma)$ is a right-covariant $(H,\alpha)$-Hom-bimodule. Thus, by using the unique projection $P_R:(\Gamma,\gamma)\to (\Gamma^{coH},\gamma|_{\Gamma^{coH}}),\:P_R(\rho)=\omega_{[0]}\cdot S(\omega_{[1]}) $ we define the linear mapping

$$\eta_{\Gamma}:H\to \Gamma^{coH},\: \eta(h):=P_R(dh),$$

for any $h\in H$, in $\widetilde{\mathcal{H}}(\mathcal{M}_k)$, for which $\eta(H)\Gamma^{coH}$. Since $\phi(dh)=dh_1\otimes h_2$, we have, for $h\in H$

$$\eta(h)=dh_1\cdot S(h_2)\: and \: dh=\eta(h_1)\cdot h_2.$$

\subsection{Bicovariant Hom-FODC }

\begin{definition}A $(H,\alpha)$-Hom-FODC $(\Gamma,\gamma)$ is said to be {\it bicovariant} if it is both left-covariant and right-covariant FODC.
\end{definition}

\begin{remark} By the Remark \ref{left-covariant-Hom-FODC2} and the Definition \ref{righ-covariant-Hom-FODC1}, a $(H,\alpha)$-Hom-FODC $(\Gamma,\gamma)$ is bicovariant if and only if there exist morphisms $\phi_L:\Gamma\to H\otimes \Gamma$ and $\phi_R:\Gamma\to \Gamma\otimes H$ in $\widetilde{\mathcal{H}}(\mathcal{M}_k)$, satisfying the equations \ref{left-coaction-on-Hom-FODC} and \ref{right-coaction-on-Hom-FODC}, respectively. So, if $(\Gamma,\gamma)$ is a bicovariant $(H,\alpha)$-Hom-FODC with Hom-coactions $\phi_L$ and $\phi_R$ satisfying \ref{left-coaction-on-Hom-FODC} and \ref{right-coaction-on-Hom-FODC} we get, for $h,g\in H$,
$$(id\otimes \phi_R)(\phi_L(h\cdot dg))=(id\otimes \phi_R)(h_1g_1\otimes h_2\cdot dg_2)=h_1g_1\otimes (h_{21}\cdot dg_{21}\otimes h_{22}g_{22}),$$
\begin{eqnarray*}(\tilde{a}\circ(\phi_L\otimes id))(\phi_R(h\cdot dg))&=&(\tilde{a}\circ(\phi_L\otimes id))(h_1\cdot dg_1\otimes h_2g_2)\\
&=&\alpha(h_{11}g_{11})\otimes(h_{12}\cdot dg_{12}\otimes \alpha^{-1}(h_2g_2))\\
&=&h_1g_1\otimes (h_{21}\cdot dg_{21}\otimes h_{22}g_{22}).\end{eqnarray*}
Thus, $(\Gamma,\gamma)$ is a bicovariant $(H,\alpha)$-Hom-bimodule and the whole structure theory of bicovariant Hom-bimodules can be applied to it.
\end{remark}
\begin{lemma}Let $(H,\alpha)$ be a monoidal Hom-Hopf algebra. Then
\begin{enumerate}
\item the linear mapping $\widetilde{Ad}_R:H\to H\otimes H$ given by
$$\widetilde{Ad}_R(h)=\alpha(h_{12})\otimes S(h_{11})\alpha^{-1}(h_{2})=\alpha(h_{21})\otimes S(\alpha^{-1}(h_1))h_{22}$$
is a right Hom-coaction of $(H,\alpha)$ on itself.
\item The linear mapping $\widetilde{Ad}_L:H\to H\otimes H$ given by
$$\widetilde{Ad}_L(h)=\alpha(h_{11})S(\alpha^{-1}(h_2))\otimes \alpha(h_{12})=\alpha^{-1}(h_1)S(h_{22})\otimes \alpha(h_{21})$$
is a left Hom-coaction of $(H,\alpha)$ on itself.
$\widetilde{Ad}_R$ and $\widetilde{Ad}_L$ are called {\it adjoint right Hom-coaction} and {\it adjoint left Hom-coaction} of $(H,\alpha)$ on itself, respectively
\end{enumerate}
\end{lemma}
\begin{proof}
\begin{enumerate}
\item  If we write $\widetilde{Ad}_R(h)=h_{[0]}\otimes h_{[1]}$ for $h\in H$, then the Hom-coassociativity follows from
\begin{eqnarray*}\alpha^{-1}(h_{[0]})\otimes \Delta(h_{[1]})&=&\alpha^{-1}(\alpha(h_{12}))\otimes \Delta(S(h_{11})\alpha^{-1}(h_{2}))\\
&=&h_{12}\otimes S(h_{112})\alpha^{-1}(h_{21})\otimes S(h_{111})\alpha^{-1}(h_{22})\\
&=&\alpha^{2}(h_{1212})\otimes S(\alpha(h_{1211}))h_{122}\otimes S(\alpha^{-1}(h_{11}))\alpha^{-2}(h_{22})\\
&=&\alpha^{2}(h_{1212})\otimes S(\alpha(h_{1211}))h_{122}\otimes \alpha^{-1}(S(h_{11})\alpha^{-1}(h_{22}))\\
&=&h_{[0][0]}\otimes h_{[0][1]} \otimes \alpha^{-1}(h_{[1]}),
\end{eqnarray*}
where in the third step we have used
$$h_{11}\otimes \alpha(h_{1211})\otimes h_{1212}\otimes \alpha^{-1}(h_{122})\otimes h_2=\alpha(h_{111})\otimes h_{112})\otimes\alpha^{-2}( h_{12})\otimes \alpha^{-2}(h_{21})\otimes\alpha( h_{22}),$$
which results from
\begin{eqnarray*} &&((id_H\otimes \tilde{a}_{H,H,H})\otimes id_H)\circ((id_H\otimes(\Delta\otimes id_H))\otimes id_H)\\
&&\circ((id_H\otimes \Delta)\otimes id_H)\circ(\Delta\otimes id_H)\circ\Delta\\
&=&(\tilde{a}_{H,H,H\otimes H}\otimes id_H)\circ((\Delta\otimes id_{H\otimes H})\otimes id_H)\circ(\tilde{a}_{H,H,H}\otimes id_H)\\
&& \circ\tilde{a}^{-1}_{H\otimes H,H,H}\circ(id_{H\otimes H}\otimes \Delta)\circ(\Delta\otimes id_H)\circ\Delta.
\end{eqnarray*}
Hom-unity condition: For any $h\in H$,
\begin{eqnarray*}h_{[0]}\varepsilon(h_{[1]})&=&\alpha(h_{12})\varepsilon(S(h_{11})\alpha^{-1}(h_{2}))=\alpha(h_{12})\varepsilon(h_{11})\varepsilon(h_{2})\\
&=&\alpha(\varepsilon(h_{11})h_{12})\varepsilon(h_{2})=h_1\varepsilon(h_{2})=\alpha^{-1}(h),
\end{eqnarray*}
and one can also easily show that $\widetilde{Ad}_R\circ \alpha=(\alpha\otimes\alpha )\circ\widetilde{Ad}_R$. Thus $\widetilde{Ad}_R$ is a right Hom-action of $(H,\alpha)$ onto itself.
\item In a similar manner, it can be proven that $\widetilde{Ad}_L$ is a left Hom-action of $(H,\alpha)$ onto itself.
\end{enumerate}
\end{proof}
With the next lemma we describe the right Hom-coaction $\phi_R$ on a left-invariant form $\omega_{\Gamma}(h)$ and the left Hom-coaction $\phi_L$ on a right-invariant form $\eta_{\Gamma}(h)$ by means of $\widetilde{Ad}_R$ and $\widetilde{Ad}_L$, respectively.

\begin{lemma}\label{coactions-on-left-right-inv-forms}For $h\in H$, we have the formulas
\begin{enumerate}
\item $\phi_R(\omega(h))=(\omega\otimes id)(\widetilde{Ad}_R(h)),$
\item $\phi_L(\eta(h))=(id\otimes \eta)(\widetilde{Ad}_L(h)).$
\end{enumerate}
\end{lemma}
\begin{proof}
\begin{enumerate}
\item For $h\in H$,
\begin{eqnarray*}\phi_R(\omega(h))&=&\Delta(S(h_1))(d\otimes id)(\Delta(h_2))\\
&=&(S(h_{12})\otimes S(h_{11}))(dh_{21}\otimes h_{22})=S(h_{12})\cdot dh_{21}\otimes S(h_{11})h_{22}\\
&=&S(\alpha(h_{121}))\cdot d(\alpha(h_{122}))\otimes S(h_{11})\alpha^{-1}(h_{2})\\
&=&\omega(\alpha(h_{12}))\otimes S(h_{11})\alpha^{-1}(h_2)=(\omega\otimes id)(\alpha(h_{12})\otimes S(h_{11})\alpha^{-1}(h_{2}))\\
&=&(\omega\otimes id)(\widetilde{Ad}_R(h)).
\end{eqnarray*}
\item Similarly, one can show that the equality $\phi_L(\eta(h))=(id\otimes \eta)(\widetilde{Ad}_L(h))$ holds.
\end{enumerate}
\end{proof}

\begin{proposition}Suppose that $(\Gamma,\gamma)$ is a left-covariant $(H,\alpha)$-Hom-FODC with associated right Hom-ideal $\mathcal{R}_{\Gamma}$. Then $(\Gamma,\gamma)$ is a bicovariant $(H,\alpha)$-Hom-FODC if and only if $\widetilde{Ad}_R(\mathcal{R})\subseteq \mathcal{R}\otimes H$, that is, $\mathcal{R}$ is $\widetilde{Ad}_R$-invariant.
\end{proposition}
\begin{proof} If $(\Gamma,\gamma)$ is a bicovariant Hom-FODC, then the equation obtained in Lemma (\ref{coactions-on-left-right-inv-forms}) holds. It implies that $\widetilde{Ad}_R(\mathcal{R})\subseteq \mathcal{R}\otimes H$ since $\mathcal{R}=\left(h\in ker\varepsilon|\:\omega(h)=0\right)$. On the contrary, suppose that $\widetilde{Ad}_R(\mathcal{R})\subseteq \mathcal{R}\otimes H$. We know that the universal Hom-FODC $\Omega^1(H)$ is bicovariant. So, by applying Lemma (\ref{coactions-on-left-right-inv-forms}) to the bicovariant Hom-FODC $\Omega^1(H)$ and using the $Ad_R$-invariance of $\mathcal{R}$, we conclude that theright Hom-action of $\Omega^1(H)$ passes to the quotient $\Omega^1(H)/\mathcal{N}$, where $\mathcal{N}:=H\omega_{\Omega^1(H)}(\mathcal(R))$, which is right-covariant. Hence, from Proposition (\ref{one-one-corresp-Hom-ideals}), $(\Gamma,\gamma)$ is right-covariant as well.
 \end{proof}
\subsection{Quantum Monoidal Hom-Lie Algebra}

Let $(\Gamma,\gamma)$ be a bicovariant $(H,\alpha)$-Hom-FODC with associated right Hom-ideal $\mathcal{R}$ and finite dimensional quantum Hom-tangent space $(\mathcal{T},\tau)$, where $\tau=\bar{\alpha}|_{\mathcal{T}}$.

We define a linear mapping $[-,-]:\mathcal{T}\otimes \mathcal{T}\to \mathcal{T}$ by setting, for $X,Y \in\mathcal{T}$,
\begin{equation}\label{quantum-Hom-Lie-bracket}[X,Y](h)=(X\otimes Y)(\widetilde{Ad}_R(h)),\: \forall h\in H.\end{equation}

$[X,Y]\in \mathcal{T}$: Indeed, since $\widetilde{Ad}_R(\mathcal{R})\subseteq \mathcal{R}\otimes H$ by the previous proposition and any element of $\mathcal{T}$ annihilates $\mathcal{R}$ by the definition of quantum Hom-tangent space, $(X\otimes Y)(\widetilde{Ad}_R(h))=0$ for all $h\in \mathcal{R}$, i.e., $[X,Y](h)=0,\:\forall h\in \mathcal{R}$. We also obtain $[X,Y](1)=0$ since $X(1)=0=Y(1)$. Thus $[X,Y]\in \mathcal{T}$. Besides, we have
\begin{eqnarray*}[\tau(X),\tau(Y)](h)&=&(X\circ\alpha^{-1}\otimes Y\circ\alpha^{-1})(\widetilde{Ad}_R(h))\\
&=&X(h_{12})Y(S(\alpha^{-1}(h_{11}))\alpha^{-2}(h_2))\\
&=&(X\otimes Y)(\widetilde{Ad}_R(\alpha^{-1}(h)))=[X,Y](\alpha^{-1}(h))\\
&=&\tau([X,Y])(h),
\end{eqnarray*}
for any $h\in H$, which means $[-,-]:\mathcal{T}\otimes \mathcal{T}\to \mathcal{T}$ is a morphism in $\widetilde{\mathcal{H}}(\mathcal{M}_k)$.

We now fix some notation. Suppose that $<\cdot,\cdot>:\mathcal{T}\times \:^{coH}\Gamma \to k$ is the bilinear form in the Proposition \ref{bilinear-form}. There exists a unique bilinear form $<\cdot,\cdot>_2\: :(\mathcal{T}\otimes \mathcal{T})\times \:^{coH}(\Gamma\otimes_H \Gamma) \to k$ defined by

\begin{equation}<X\otimes Y,u\otimes v>_2\: =<X,u><Y,v>\end{equation}

for $X,Y\in \mathcal{T}$ and $u,v\in \: ^{coH}\Gamma$, which is nondegenerate as the bilinear form $<\cdot,\cdot>$ is. If we put $B: \Gamma\otimes_H \Gamma\to\Gamma\otimes_H \Gamma $ for the Woronowicz' braiding, then, for $h,g\in H$, we compute

 \begin{eqnarray}\label{action-of-Woronowicz-braiding}
 \nonumber B(\omega(h)\otimes_H\omega(g))&=&\gamma(\omega(\alpha(g_{12})))\otimes_H\gamma^{-1}(\omega(h))\lhd(S(g_{11})\alpha^{-1}(g_2))\\
 \nonumber &=&\omega(\alpha^{2}(g_{12}))\otimes_H \gamma^{-1}(\omega(h)\lhd (S(\alpha(g_{11}))g_2))\\
 \nonumber &=&\omega(\alpha^{2}(g_{12}))\otimes_H \gamma^{-1}(\omega(\bar{h}(S(\alpha(g_{11}))g_2)))\\
 &=&\omega(\alpha^{2}(g_{12}))\otimes_H \omega(\overline{\alpha^{-1}(h)}(S(g_{11})\alpha^{-1}(g_2))).
 \end{eqnarray}
 With respect to the nondegenerate bilinear form $<\cdot,\cdot>_2$, we define the transpose $B^{t}$ of $B$ as a linear map acting on $\mathcal{T}\otimes \mathcal{T}$ such that
$$<B^{t}(X\otimes Y),u\otimes v>_2\: =<X\otimes Y,B(u\otimes v)>_2.$$

We now recall that the dual  monoidal Hom-Hopf algebra $(H^{\circ},\alpha^{\circ})$ of $(H,\alpha)$ consists of functionals $f\in H'$ for which $\Delta_{H^{\circ}}(f)=f_1\otimes f_2 \in H'\otimes H'$ and the Hom-counit is given by $\varepsilon_{H^{\circ}}(f)=f(1_H)$. Since, also $\Delta(f)(h\otimes g):=f(hg)$ for $\Delta(f)\in (H\otimes H)' $ and $h,g \in H$, we have $f(hg)=f_1(h)f_2(g)$. $\alpha^{\circ}$ is given by $\alpha^{\circ}(f)=f\circ \alpha^{-1}$ for $f\in H^{\circ}$. Hom-multiplication $m_{H^{\circ}}$ is the convolution, i.e., $m_{H^{\circ}}(f\otimes f')(h)=(ff')(h)=f(h_1)f'(h_2)$ for $f,f'\in H'$, $h\in H$ and the Hom-unit is $\varepsilon_H$. The antipode is given by $S(f)(h)=f(S(h))$ for $f\in H^{\circ}$ and $h\in H$. Since we assumed that $\mathcal{T}_{\Gamma}$ is finite dimensional, $\mathcal{T}_{\Gamma}$ is contained in $H^{\circ}$. Thus we have the following theorem in $(H^{\circ},\alpha^{\circ})$.

\begin{theorem}For any $X,Y,Z \in \mathcal{T}_{\Gamma}$ we have
\begin{enumerate}
\item $[X,Y]=\widetilde{ad}_R(Y)(X)=XY-m_{H^{\circ}}(B^{t}(X\otimes Y))$.
\item Let $\chi=\sum_iX_i\otimes Y_i$ for $X_i,Y_i\in \mathcal{T}$ such that $B^{t}(\chi)=\chi$, then $\sum_i[X_i,Y_i]=0$.
\item $[\tau(X),[Y,\tau^{-1}(Z)]]=[[X,Y],Z]-\sum_i[[X,\tau^{-1}(Z_i)],\tau(Y_i)]$, where $Y_i,Z_i\in \mathcal{T}$ such that $B^{t}(Y\otimes Z)=\sum_iZ_i\otimes Y_i$.
\end{enumerate}
\end{theorem}

\begin{proof}
\begin{enumerate}
\item For $h\in H$,
\begin{eqnarray*}\widetilde{ad}_R(Y)(X)(h)&=&((S(Y_1)\tau^{-1}(X))\tau(Y_2))(h)\\
&=&(S(Y_1)\tau^{-1}(X))(h_1)\tau(Y_2))(h_2)=S(Y_1)(h_{11})\tau^{-1}(X)(h_{12})\tau(Y_2))(h_2)\\
&=&Y_1(S(h_{11}))X(\alpha(h_{12}))Y_2(\alpha^{-1}(h_2))=X(\alpha(h_{12}))Y_1(S(h_{11}))Y_2(\alpha^{-1}(h_2))\\
&=&X(\alpha(h_{12}))Y(S(h_{11})\alpha^{-1}(h_2))=(X\otimes Y)(\alpha(h_{12})\otimes S(h_{11})\alpha^{-1}(h_2))\\
&=&(X\otimes Y)(\widetilde{Ad}_R(h))=[X,Y](h),
\end{eqnarray*}
which gives us the first equality. If we set the finite sum for $B^{t}(X\otimes Y)=\sum_iY_i\otimes X_i$ with $X_i,Y_i\in \mathcal{T}$, then, for any $h,g\in H$,
\begin{eqnarray*} B^{t}(X\otimes Y)(h\otimes g)&=&\sum_iY_i(h)X_i(g)=\sum_i<\tau^{-1}(Y_i),\omega(h)><\tau^{-1}(X_i),\omega(g)>\\
&=&<\sum_i\tau^{-1}(Y_i)\otimes \tau^{-1}(X_i),\omega(h)\otimes_H \omega(g)>_2\\
&=&<B^{t}(\tau^{-1}(X)\otimes \tau^{-1}(Y)),\omega(h)\otimes_H \omega(g)>_2\\
&=&<\tau^{-1}(X)\otimes \tau^{-1}(Y),B(\omega(h)\otimes_H \omega(g))>_2\\
&=&<\tau^{-1}(X)\otimes\tau^{-1}(Y),\omega(\alpha^{2}(g_{12}))\otimes_H\omega(\overline{\alpha^{-1}(h)}(S(g_{11})\alpha^{-1}(g_2)))>_2\\
&=&<\tau^{-1}(X),\omega(\alpha^{2}(g_{12}))><\tau^{-1}(Y),\omega(\overline{\alpha^{-1}(h)}(S(g_{11})\alpha^{-1}(g_2)))>\\
&=&X(\alpha^{2}(g_{12}))Y(\overline{\alpha^{-1}(h)}(S(g_{11})\alpha^{-1}(g_2)))\\
&=&X(\alpha^{2}(g_{12}))Y_1(\overline{\alpha^{-1}(h)})Y_2(S(g_{11})\alpha^{-1}(g_2))\\
&=&Y_1(\overline{\alpha^{-1}(h)})X(\alpha^{2}(g_{12}))Y_{21}(S(g_{11}))Y_{22}(\alpha^{-1}(g_2))\\
&=&\overline{\tau(Y_1)}(h)S(Y_{21})(g_{11})X(\alpha^{2}(g_{12}))Y_{22}(\alpha^{-1}(g_2))\\
&=&\overline{\tau(Y_1)}(h)(S(Y_{21})\tau^{-2}(X))(g_1)\tau(Y_{22})(g_2)\\
&=&\overline{\tau(Y_1)}(h)[(S(Y_{21})\tau^{-2}(X))\tau(Y_{22})](g)\\
&=&(\overline{\tau(Y_1)}\otimes\widetilde{ad}_R(Y_2)(\tau^{-1}(X)))(h\otimes g),
\end{eqnarray*}
where in the sixth equality we have used the equation \ref{action-of-Woronowicz-braiding}. So, we have $B^{t}(X\otimes Y)=\overline{\tau(Y_1)}\otimes\widetilde{ad}_R(Y_2)(\tau^{-1}(X))$. Hence, we make the following computation

\begin{eqnarray*}m_{H'}(B^{t}(X\otimes Y))&=&\overline{\tau(Y_1)}\widetilde{ad}_R(Y_2)(\tau^{-1}(X))\\
&=&(\tau(Y_1)-\varepsilon_{H^{\circ}}(\tau(Y_1))1_{H^{\circ}})\widetilde{ad}_R(Y_2)(\tau^{-1}(X))\\
&=&\tau(Y_1)[(S(Y_{21})\tau^{-2}(X))\tau(Y_{22})]-(\varepsilon_{H^{\circ}}(Y_1)1_{H^{\circ}})[(S(Y_{21})\tau^{-2}(X))\tau(Y_{22})]\\
&=&\tau(Y_1)[S(\tau(Y_{21}))(\tau^{-2}(X)Y_{22})]-(\varepsilon_{H^{\circ}}(Y_1)1_{H^{\circ}})((S(Y_{21})\tau^{-2}(X))\tau(Y_{22}))\\
&=&(Y_1S(\tau(Y_{21})))\tau(\tau^{-2}(X)Y_{22})-(\varepsilon_{H^{\circ}}(Y_1)1_{H^{\circ}})((S(Y_{21})\tau^{-2}(X))\tau(Y_{22}))\\
&=&(\tau(Y_{11})S(\tau(Y_{12})))(\tau^{-1}(X)Y_2)-1_{H^{\circ}}([S(\varepsilon_{H^{\circ}}(Y_{11})Y_{12})\tau^{-2}(X)]Y_2)\\
&=&(\varepsilon_{H^{\circ}}(Y_1)1_{H^{\circ}})(\tau^{-1}(X)Y_2)-(S(Y_1)\tau^{-1}(X))\tau(Y_2)\\
&=&1_{H^{\circ}}(\tau^{-1}(XY))-\widetilde{ad}_R(Y)(X)\\
&=&XY-\widetilde{ad}_R(Y)(X),
\end{eqnarray*}
that is, we get $\widetilde{ad}_R(Y)(X)=XY-m_{H'}(B^{t}(X\otimes Y))$.

\item  It immediately follows from (1) that $\sum_i[X_i,Y_i]=\sum_iX_iY_i-m_{H'}(B^{t}(\chi))=\sum_iX_iY_i-\sum_iX_iY_i=0$.

\item Let us first set $[X,Y]=\widetilde{ad}_R(Y)(X)=X\lhd Y$. Then,
$$[[X,Y],Z]=[X,Y]\lhd Z=(X\lhd Y)\lhd Z=\tau(X)\lhd (Y \tau^{-1}(Z))=[\tau(X),Y\tau^{-1}(Z)].$$
Since, by (1), $YZ=[Y,Z]+\sum_iZ_iY_i$ for $B^{t}(Y\otimes Z)=\sum_iZ_i\otimes Y_i$, we have

\begin{eqnarray*}[[X,Y],Z]&=&[\tau(X),Y\tau^{-1}(Z)]=[\tau(X),[Y,\tau^{-1}(Z)]+\sum_i\tau^{-1}(Z_i)Y_i]\\
&=&[\tau(X),[Y,\tau^{-1}(Z)]]+\sum_i[\tau(X),\tau^{-1}(Z_i)Y_i]\\
&=&[\tau(X),[Y,\tau^{-1}(Z)]]+\sum_i[[X,\tau^{-1}(Z_i)], \tau(Y_i)],
\end{eqnarray*}
that is, $[\tau(X),[Y,\tau^{-1}(Z)]]=[[X,Y],Z]-\sum_i[[X,\tau^{-1}(Z_i)], \tau(Y_i)]$ holds.
\end{enumerate}
\end{proof}

\begin{remark}If we take the braiding $B$ as the flip operator, then $B^{t}$ is the flip on $\mathcal{T}\otimes \mathcal{T}$ by its definition. In this case,
we obtain
$$[X,Y]=XY-YX, \qquad [X,Y]+[Y,X]=0, \forall X,Y\in \mathcal{T}$$
and
\begin{eqnarray*}[\tau(X),[Y,\tau^{-1}(Z)]]&=&[[X,Y],Z]-[[X,\tau^{-1}(Z)],\tau(Y)]=-[Z,[X,Y]]+[\tau(Y),[X,\tau^{-1}(Z)]]\\
&=&-[Z,[X,Y]]-[\tau(Y),[\tau^{-1}(Z),X]].
\end{eqnarray*}
Then, by replacing $Z$ with $\tau(Z)$ in the above equality, we get
$$[\tau(X),[Y,Z]]+[\tau(Y),[Z,X]]+[\tau(Z),[X,Y]]=0,$$
which is the Hom-Jacobi identity. In the above theorem, items $(2)$ and $(3)$ are the quantum versions of the antisymmetry relation and the Hom-Jacobi identity. Therefore, $(\mathcal{T}_{\Gamma},\tau)$ is called the {\it quantum Hom-Lie algebra} of the bicovariant $(H,\alpha)$-Hom-FODC.

\end{remark}

\section{Acknowledgments}
The author would like to thank Professor Christian Lomp for his valuable suggestions. This research was funded by the European Regional Development Fund through the programme COMPETE and by the Portuguese Government through the FCT- Fundação para a Ciência e a Tecnologia under the project PEst-C/MAT/UI0144/2013. The author was supported by the grant SFRH/BD/51171/2010.

\end{document}